\documentclass[10pt,a4paper]{article}
\usepackage[a4paper]{geometry}
\usepackage{amssymb,latexsym,amsmath,amsfonts,amsthm}
\usepackage{graphicx,comment}
\usepackage{epsfig}
\usepackage{tikz}
\usepackage[T1]{fontenc}

\newcommand{\Fkr}{\widetilde{F}_{k}^{(\rho)}}
\newcommand{\Fk}{F_{k}^{(\rho)}}

\newcommand{\Fkm}{F_{k-1}^{(\rho)}}

\newcommand{\Fkp}{F_{k+1}^{(\rho)}}

\newtheorem{theorem}{Theorem}[section]
\newtheorem{lemma}[theorem]{Lemma}
\newtheorem{proposition}[theorem]{Proposition}

\newtheorem{corollary}[theorem]{Corollary}

\theoremstyle{definition}
\newtheorem{definition}[theorem]{Definition}

\theoremstyle{remark}

\newtheorem{remark}[theorem]{Remark}

\numberwithin{equation}{section}

\title{Nikishin systems on star-like sets: Ratio asymptotics of the associated multiple orthogonal polynomials, II}

\date{\today}

\author{Abey L\'{o}pez-Garc\'{i}a\footnotemark[1] \quad Guillermo L\'{o}pez Lagomasino\footnotemark[2]}

\begin{document}

\maketitle

\renewcommand{\thefootnote}{\fnsymbol{footnote}}
\footnotetext[1]{Department of Mathematics, University of Central Florida, 4393 Andromeda Loop North, Orlando, FL 32816, USA. email: abey.lopez-garcia\symbol{'100}ucf.edu.} \footnotetext[2]{Departamento de
Matem\'{a}ticas, Universidad Carlos III de Madrid, Avda.~Universidad 30, 28911, Legan\'{e}s, Madrid, Spain. email:
lago\symbol{'100}math.uc3m.es.\\ \indent  The second author received partial support from the Spanish Ministerio de Ciencia, Innovaci\'on y Universidades through grant PGC2018-096504-B-C33.
}

\begin{abstract} In this paper we continue the investigations initiated in \cite{LopLopstar} on ratio asymptotics of multiple orthogonal polynomials and functions of the second kind associated with Nikishin systems on star-like sets. We describe in detail the limiting functions found in \cite{LopLopstar}, expressing them in terms of certain conformal mappings defined on a compact Riemann surface of genus zero. We also express the limiting values of the recurrence coefficients, which are shown to be strictly positive, in terms of certain values of the conformal mappings. As a consequence, the limits depend exclusively on the location of the intervals determined by the supports of the measures that generate the Nikishin system.
\smallskip


\smallskip


\end{abstract}

\section{Introduction and statement of main results}

Multiple orthogonal polynomials (MOP) and their asymptotic properties have received considerable attention in the last three decades, partly due to their applicability in different fields. The so called Nikishin systems of measures introduced in \cite{Nik} play a central role in many of these studies. Some of the basic questions involve uniqueness of the MOP \cite{FidLop}, convergence of the corresponding Hermite-Pad\'e approximants \cite{BusLop}, $n$-th root \cite{GonRakSor}, ratio \cite{AptLopRocha} (see also \cite{LopLopratio}), and strong \cite{Apt,LopVan} asymptotics of sequences of MOP. We have limited to a short list of significant contributions, see also reference lists in \cite{LopLopstar, LopMin}.

This paper is devoted to the study of the ratio asymptotic behavior of MOP associated with Nikishin systems of measures on star-like sets and of the limit behavior of the coefficients in the recurrence relation they satisfy. It is a  continuation of the investigations in \cite{LopLopstar, LopMin}. We improve the results in \cite{LopLopstar} by giving a detailed expression of the limiting functions that describe the ratio asymptotics and the limiting values of the recurrence coefficients. See also \cite{Lop} for an account of corresponding results in the case $p=2$.

The interest in the study of MOP on star-like sets has its origin in the study of Faber polynomials associated with hypocycloidal domains \cite{EierVarga,HeSaff} and the asymptotic and spectral properties of polynomials generated by high order three-term recurrence relations \cite{AptKalSaff, AptKalIse, DelLop}. Recently, MOP on star-like sets associated with Angelesco systems or classical type weights  have been studied in \cite{LeuVan1, LeuVan2, LouVan}.

Let $p\geq 2$ be an integer, and let
\[
S_{\pm}:=\{z\in\mathbb{C}: z^{p+1}\in\mathbb{R}_{\pm}\},\qquad \mathbb{R}_{+}=[0,+\infty),\quad \mathbb{R}_{-}=(-\infty,0].
\]
We construct $p$ finite stars contained in $S_{\pm}$ as follows:
\begin{align*}
\Gamma_{j} & :=\{z\in\mathbb{C}: z^{p+1}\in[a_{j},b_{j}]\},\qquad \quad 0\leq
j\leq p-1,
\end{align*}
where
\[
0\leq a_{j}<b_{j}<\infty,\quad j\equiv 0\mod 2,
\]
\[
-\infty< a_{j}<b_{j}\leq 0, \quad j\equiv 1\mod 2,
\]
thus $\Gamma_{j}\subset S_{+}$ if $j$ is even, and $\Gamma_{j}\subset S_{-}$
if $j$ is odd. We assume throughout that $\Gamma_{j}\cap
\Gamma_{j+1}=\emptyset$ for all $0\leq j\leq p-2$.

We define now a Nikishin system on $(\Gamma_{0},\ldots,\Gamma_{p-1})$. For each
$0\leq j\leq p-1$,  let $\sigma_{j}$ denote a positive, rotationally invariant
measure on $\Gamma_{j}$, with infinitely many points in its support. These will
be the measures generating the Nikishin system. Let
\[
\widehat{\sigma}_j(z):=\int\frac{d\sigma_j(t)}{z-t}
\]
denote the Cauchy transform of  $\sigma_j$.
We define the measure $\langle \sigma_{0},\ldots,\sigma_{N}\rangle$ by the following
recursive procedure.
For $N=0$, $\langle \sigma_{0}\rangle:= \sigma_{0}$, for $N=1$,
\[
d\langle \sigma_{0},\sigma_{1}\rangle (z):= \widehat{\sigma}_{1}(z)\,d\sigma_{0}(z),
\]
and for $N>1$,
\[
\langle \sigma_{0},\ldots,\sigma_{N}\rangle := \langle \sigma_{0},\langle
\sigma_{1},\ldots,\sigma_{N}\rangle\rangle.
\]

We define the Nikishin system
$(s_{0},\ldots,s_{p-1})=\mathcal{N}(\sigma_0,\ldots,\sigma_{p-1})$ generated by
the vector of
$p$ measures $(\sigma_0,\ldots,\sigma_{p-1})$ by setting
\begin{equation}\label{def:sj}
s_{j}:=\langle \sigma_{0},\ldots, \sigma_{j}\rangle, \qquad 0\leq j\leq p-1.
\end{equation}
Notice that the measures $s_{j}$ are all supported on the first star $\Gamma_{0}$.

\begin{definition}\label{def:MOP}
Let $(Q_{n})_{n=0}^\infty$ be the sequence of monic polynomials of lowest
degree that satisfy the following non-hermitian orthogonality conditions:
\begin{equation}\label{orthog:Qn}
\int_{\Gamma_{0}} Q_{n}(z)\,z^{l}\,d s_{j}(z)=0,\qquad l=0,\ldots,\left\lfloor
\frac{n-j-1}{p}\right\rfloor,\qquad 0\leq j\leq p-1,
\end{equation}
where the measures $s_{j}$ are those in \eqref{def:sj}, and $\lfloor\cdot\rfloor$ denotes the floor function.\end{definition}
In more detail, \eqref{orthog:Qn} asserts that the polynomial $Q_{n}$, where $n= mp+r$, $0\leq r \leq p-1$,
satisfies the orthogonality conditions
\begin{align*}
\int_{\Gamma_{0}} Q_{mp+r}(z)\,z^{l}\,ds_{j}(z) & =0,\qquad l=0,\ldots,m-1,\quad
0\leq j\leq p-1,\\
\int_{\Gamma_{0}} Q_{mp+r}(z)\,z^{m}\,ds_{j}(z) & =0,\qquad 0\leq j\leq
r-1.
\end{align*}

Some algebraic properties of these polynomials were proved in \cite[Propositions 2.16, 3.1, Theorem 3.5, Corollary 3.6]{LopMin}. For our purpose, the most significant are:

\begin{itemize}
\item[1)] For each $n\geq 0$, the polynomial $Q_{n}$ has maximal degree $n$.
\item[2)] If $n\equiv \ell \mod (p+1)$, $0\leq \ell\leq p$, then there exists
a monic polynomial $\mathcal{Q}_{d}$ of degree $d=\frac{n-\ell}{p+1}$ such that
\begin{equation}\label{eq:decompQn}
Q_{n}(z)=z^{\ell} \mathcal{Q}_{d}(z^{p+1}),
\end{equation}
where the zeros of $\mathcal{Q}_{d}$ are all simple and located in $(a_{0},b_{0})$. In particular,
the zeros of $Q_{n}$ are located in the star-like set $S_{+}$.
\item[3)] The polynomials $Q_{n}$ satisfy the following three-term recurrence relation of order $p+1$:
\begin{equation}\label{threetermrec}
z\,Q_{n}(z)=Q_{n+1}(z)+a_{n}\,Q_{n-p}(z),\qquad n\geq p,\qquad
a_{n}>0,
\end{equation}
where
\[
Q_{\ell}(z)=z^{\ell},\qquad \ell=0,\ldots,p.
\]
(Here, there is an abuse of notation since above we denoted by $a_0,\ldots,a_{p-1}$ the left end points of some intervals on the real line. From the context, we are sure this will cause no confusion in the text.)
\item[4)] For every $n\geq p+1$, the non-zero roots of the polynomials $Q_{n}$ and $Q_{n+1}$ interlace on $\Gamma_{0}$.
\end{itemize}

Recurrences of the form \eqref{threetermrec}, with $a_n>0$ for all $n$, were studied in \cite{AptKalIse,AptKalSaff,Ben,DelLop}. In \cite{AptKalIse,AptKalSaff,DelLop}, Favard type theorems were obtained showing that the generated polynomials satisfy multiple orthogonality relations with respect to measures with common support on a star-like set.

\begin{definition}\label{definitionPsi}
The functions of the second kind are defined as follows. Set $\Psi_{n,0}:=Q_{n}$, and let
\[
\Psi_{n,k}(z):=\int_{\Gamma_{k-1}}\frac{\Psi_{n,k-1}(t)}{z-t}\,d\sigma_{k-1}(t),
\qquad k=1,\ldots,p.
\]
\end{definition}

These functions satisfy the same three-term recurrence relation \eqref{threetermrec} and, therefore, they also play a central role in the asymptotic analysis.

In \cite[Theorem 5.3, Corollary 5.4]{LopMin}, under appropriate assumptions on the generating measures, the asymptotic zero distribution and n-th root asymptotics of the sequences $(Q_n)_{n=0}^\infty$ and $(\Psi_{n,k})_{n=0}^{\infty}$, $k=1,\ldots,p$, were given in terms of the solution of a vector equilibrium problem for the logarithmic potential.

In \cite{LopLopstar}, the goal was to obtain an extension of Rakhmanov's theorem on ratio asymptotics of orthogonal polynomials \cite{Rak1, Rak2} (for simplifications and alternative proofs of this classical result, see also \cite{MNT1,MNT2,Nev1,Nev2,Rak3}) similar to the one given in \cite{AptLopRocha} for Nikishin systems on the real line. Indeed, it was shown, see \cite[Corollary 3.3]{LopLopstar}, that $\sigma_j' > 0$  a.e. on $\Gamma_j$, $j=0,\ldots,p-1,$ implies that for each $\rho$, $0\leq \rho \leq p(p+1) - 1,$ there exists
\[ \lim_{\lambda\to \infty} \frac{Q_{\lambda p(p+1) + \rho +1}}{Q_{\lambda p(p+1) + \rho}}
\]
uniformly on compact subsets of $\mathbb{C} \setminus (\Gamma_0 \cup \{0\})$. Somewhat surprisingly, the limits exist over a period $p(p+1)$. In \cite{AptLopRocha}, for Nikishin systems on the real line generated by $p$ measures, it was shown that ratio asymptotics holds with period $p$ and the limiting functions were described in terms of  certain algebraic functions defined on a Riemann surface of genus zero with $p+1$ sheets; as shown in \cite[Corollary 1.3]{AKLR}, this implies the $p$ periodic limits of the coefficients in the $(p+2)$-term recurrence relation satisfied by the Nikishin multiple orthogonal polynomials (see also \cite[Theorem 1.1]{DelLopLop}, where algebraic relations between these limits are described). An analogous question in the star setting remained unanswered in \cite{LopLopstar}. In \cite[Corollary 3.3]{LopLopstar} it was also proved the existence of
\begin{equation}\label{eq:limarho}
\lim_{\lambda\to \infty} a_{\lambda p(p+1)+\rho}=a^{(\rho)}.
\end{equation}
Here, we show that these limits are non-zero, we give different expressions for them, and describe some linear relations that they satisfy.

The asymptotic formulae that we obtain in this work can all be expressed in terms of certain conformal mappings defined on a compact Riemann surface of genus zero. These formulae show that under the general assumptions on the generating measures of the Nikishin system considered in our previous work \cite{LopLopstar}, the ratio asymptotic quantities obtained only depend, as expected, on the underlying Riemann surface whose structure is determined by the supports of the generating measures.

Before we state our main results, we define the Riemann surface and conformal mappings we will work with.

Throughout the rest of the paper, we will occasionally write $\Delta_{k}:=[a_{k},b_{k}]$, $0\leq k\leq p-1$. Let $\mathcal{R}$ denote the compact Riemann surface
\[
\mathcal{R}=\overline{\bigcup_{k=0}^{p}\mathcal{R}_{k}}
\]
formed by the $p+1$ consecutively "glued" sheets
\[
\mathcal{R}_{0}:=\overline{\mathbb{C}}\setminus\Delta_{0},\qquad \mathcal{R}_{k}:=\overline{\mathbb{C}}\setminus(\Delta_{k-1}\cup\Delta_{k}),\quad k=1,\ldots,p-1,\qquad \mathcal{R}_{p}:=\overline{\mathbb{C}}\setminus\Delta_{p-1},
\]
where the upper and lower banks of the slits of two neighboring sheets are identified. This surface is of genus zero. For this and other notions of Riemann surfaces as well as meromorphic functions defined on them we recommend \cite{Mir}.

Let $\pi: \mathcal{R} \longrightarrow \overline{\mathbb{C}}$ be the canonical projection from $\mathcal{R}$ to $\overline{\mathbb{C}}$ and denote by $z^{(k)}$ the point on  $\mathcal{R}_k$ satisfying $\pi(z^{(k)}) = z$, $z \in \overline{\mathbb{C}}$. For a fixed $l\in\{1,\ldots,p\}$, let $\varphi^{(l)}:\mathcal{R}\longrightarrow\overline{\mathbb{C}}$ denote a conformal mapping whose divisor consists of one simple zero at the point $\infty^{(0)}\in\mathcal{R}_{0}$ and one simple pole at the point $\infty^{(l)}\in\mathcal{R}_{l}$. This mapping exists and is uniquely determined up to a multiplicative constant. Denote the branches of $\varphi^{(l)}$ by
\[ \varphi_k^{(l)}(z) := \varphi^{(l)}(z^{(k)}), \qquad k= 0,\ldots,p, \qquad z^{(k)} \in \mathcal{R}_{k}.\]
From the properties of $\varphi^{(l)}$, we have
\begin{equation}\label{divisorcond}
\varphi_0^{(l)}(z)=C_{1,l}/z+O(1/z^{2}),\,\,\,z\rightarrow\infty,\qquad \varphi_l^{(l)}(z)=C_{2,l}\,z+O(1),\,\,\,z\rightarrow\infty,
\end{equation}
where $C_{1,l}$, $C_{2,l}$ are non-zero constants.

It is well known and easy to verify that the function $\prod_{k=0}^{p}\varphi_{k}^{(l)}$ admits an analytic continuation to the whole extended plane $\overline{\mathbb{C}}$ without singularities; therefore, it is constant. Multiplying $\varphi^{(l)}$ if necessary by a suitable non-zero constant, we may assume that $\varphi^{(l)}$ satisfies the conditions
\[
\prod_{k=0}^{p}\varphi_{k}^{(l)} = C, \qquad |C| = 1, \qquad C_{1,l} > 0.
\]
Let us show that with this normalization, $C$ is either $+1$ or $-1$.

Indeed, for a point $z^{(k)} \in \mathcal{R}_k$ on the Riemann surface we define its conjugate $\overline{z^{(k)}} := \overline{z}^{(k)}$. Now, let $\overline{\varphi}^{(l)}: \mathcal{R} \longrightarrow \overline{\mathbb{C}}$ be the function defined by $\overline{\varphi}^{(l)}(\zeta):= \overline{\varphi^{(l)}(\overline{\zeta})}$. It is easy to verify that $\overline{\varphi}^{(l)}$ is a conformal mapping of $\mathcal{R}$ onto $\overline{\mathbb{C}}$ with the same divisor as $\varphi^{(l)}$. Therefore, there exists a constant $c$ such that $\overline{\varphi}^{(l)} = c \varphi^{(l)}$. The corresponding branches satisfy the relations
\[ \overline{\varphi}_k^{(l)}(z) = \overline{\varphi_k^{(l)}(\overline{z})} = c {\varphi}_k^{(l)}(z), \qquad k=0,\ldots,p.
\]
Comparing the Laurent expansions at $\infty$ of $\overline{\varphi_0^{(l)}(\overline{z})}$ and $c {\varphi}_0^{(l)}(z)$, using the fact that $C_{1,l} >0$, it follows that $c = 1$. Then
\[   {\varphi}_k^{(l)}(z) = \overline{\varphi_k^{(l)}(\overline{z})}, \qquad k=0,\ldots,p.
\]
This in turn implies that for each $k=0,\ldots, p,$ all the coefficients, in particular the leading one, of the Laurent expansion at infinity of $ {\varphi}_k^{(l)}$ are real numbers. Obviously, $C$ is the product of these leading coefficients. Therefore, $C$ is real, and $|C|=1$ implies that $C$ equals $1$ or $-1$ as claimed. So, we can assume in the following that
\begin{equation}\label{normconfmap}
\prod_{k=0}^{p}\varphi_{k}^{(l)}\equiv \pm 1,\qquad C_{1,l}>0.
\end{equation}
It is easy to see that conditions \eqref{divisorcond} and \eqref{normconfmap} determine $\varphi^{(l)}$ uniquely. In this paper, we will use the notation
\begin{equation}\label{def:omegal}
\omega_{l}:=C_{1,l}=\lim_{z\rightarrow\infty} z \varphi_{0}^{(l)}(z)
\end{equation}
for the constant $C_{1,l}$ in \eqref{divisorcond}.

We can now state the main results of this paper.

\begin{theorem}\label{theo:main:1}
Assume that for each $k=0,\ldots,p-1$, the measure $\sigma_{k}$  has positive Radon-Nikodym derivative with respect to linear Lebesgue measure a.e. on $\Gamma_k$. The following formulas hold, uniformly on compact subsets of the indicated regions:
\begin{itemize}
\item[1)] For each fixed $0\leq \rho\leq p(p+1)-1$,
\begin{equation}\label{eq:ratioasympQ}
\lim_{\lambda\rightarrow\infty}\frac{Q_{\lambda p (p+1)+\rho+1}(z)}{Q_{\lambda p (p+1)+\rho}(z)}=\frac{z}{1+a^{(\rho)}\,\omega_{l}^{-1}\,\varphi_{0}^{(l)}(z^{p+1})},\qquad z\in\mathbb{C}\setminus(\Gamma_{0}\cup\{0\}),
\end{equation}
where $l=l(\rho)$ is the integer satisfying the conditions $1\leq l\leq p$ and $(l-1)\equiv \rho \mod p$, and $\omega_{l}$ is defined in \eqref{def:omegal}. Convergence takes place in $\mathbb{C}\setminus\Gamma_{0}$ if $\rho\not\equiv p \mod (p+1)$.

\item[2)] For each fixed $0\leq \rho\leq p(p+1)-1$ and $1\leq k\leq p$,
\begin{equation}\label{eq:ratioasympPsink}
\lim_{\lambda\rightarrow\infty}\frac{\Psi_{\lambda p(p+1)+\rho+1,k}(z)}{\Psi_{\lambda p(p+1)+\rho,k}(z)}=\frac{z}{1+a^{(\rho)}\,\omega_{l}^{-1}\,\varphi_{k}^{(l)}(z^{p+1})},\qquad z\in\mathbb{C}\setminus(\Gamma_{k-1}\cup\Gamma_{k}\cup\{0\}),
\end{equation}
with $\omega_{l}$ and $l=l(\rho)$ as in 1), and $\Gamma_{p}=\emptyset$.
\end{itemize}
\end{theorem}

The following result concerns properties of the limiting values $a^{(\rho)}$ in \eqref{eq:limarho}. In the statement of the result and throughout the rest of the paper, we understand that the values $(a^{(\rho)})_{\rho=0}^{p(p+1)-1}$ are continued periodically in $\mathbb{Z}$ with period $p(p+1)$, so that $a^{(\rho)}=a^{(\rho+p(p+1))}$ for all $\rho\in\mathbb{Z}$.

\begin{theorem}\label{theo:main:2}
Assume that for each $k=0,\ldots,p-1$, the measure $\sigma_{k}$  has positive Radon-Nikodym derivative with respect to linear Lebesgue measure a.e. on $\Gamma_k$.
The following properties stated in 1)--4) below hold for each $0\leq \rho\leq p(p+1)-1$:
\begin{itemize}
\item[1)] The limit in \eqref{eq:limarho} is non-zero, i.e., $a^{(\rho)}>0$.
\item[2)] The set of $p$ values $\{a^{(\rho+m(p+1))}\}_{m=0}^{p-1}$ is formed by distinct quantities.
\item[3)] The following relation holds:
\[
\sum_{i=\rho}^{\rho+p-1} a^{(i)}=\sum_{i=\rho+p+1}^{\rho+2p} a^{(i)}.
\]
\item[4)] We have
\begin{equation}\label{eq:descrip:arho}
a^{(\rho)}=-\frac{\omega_{l}}{\varphi_{k}^{(l)}(0)}
\end{equation}
where $(k,l)=(k(\rho),l(\rho))$ is the unique pair of integers satisfying the conditions $0\leq k\leq p$, $\rho\equiv (k-1) \mod (p+1)$, and $1\leq l\leq p$, $\rho\equiv (l-1) \mod p$, and $\omega_{l}$ is the positive constant defined in \eqref{def:omegal}.
\item[5)] Assume that $0\in\Delta_{k}$ for some $0\leq k\leq p-1$. Then, for any $0\leq \rho\leq p(p+1)-1$ such that $\rho\equiv (k-1) \mod (p+1)$, we have $a^{(\rho-p)}=a^{(\rho)}$. If $0\notin\Delta_{k}$ for all $0\leq k\leq p-1,$ then for any $0\leq \rho\leq p(p+1)-1$, the set of $p+1$ values $\{a^{(\rho+mp)}\}_{m=0}^{p}$ is formed by distinct quantities.
\end{itemize}
\end{theorem}

Observe that the function $\eta^{(\rho)}:\mathcal{R}\longrightarrow \overline{\mathbb{C}}$ defined by
\begin{equation}\label{def:etarhoconf}
\eta^{(\rho)} :=\frac{1}{1+a^{(\rho)}\,\omega_{l(\rho)}^{-1}\,\varphi^{(l(\rho))} }
\end{equation}
is conformal, as it is the composition of $\varphi^{(l(\rho))}$ with the fractional linear transformation $w\mapsto (1+a^{(\rho)}\,\omega_{l(\rho)}^{-1}\,w)^{-1}$. As a consequence of \eqref{eq:descrip:arho} and the definition of $\varphi^{(l(\rho))}$, the function $\eta^{(\rho)}:\mathcal{R}\longrightarrow\overline{\mathbb{C}}$ is characterized as the unique conformal mapping with a simple zero at $\infty^{(l(\rho))}$, a simple pole at $0^{(k(\rho))}$, and satisfying $\eta^{(\rho)}(\infty^{(0)})=1$. Then, \eqref{eq:ratioasympQ} and \eqref{eq:ratioasympPsink} take the simpler form
\begin{align*}
\lim_{\lambda\rightarrow\infty}\frac{Q_{\lambda p(p+1)+\rho+1}(z)}{Q_{\lambda p(p+1)+\rho}(z)} & =z \eta_{0}^{(\rho)}(z^{p+1}),\\
\lim_{\lambda\rightarrow\infty}\frac{\Psi_{\lambda p(p+1)+\rho+1,k}(z)}{\Psi_{\lambda p(p+1)+\rho,k}(z)} & =z \eta_{k}^{(\rho)}(z^{p+1}),\quad 1\leq k\leq p,
\end{align*}
where $\eta^{(\rho)}_{k}(z):=\eta^{(\rho)}(z^{(k)})$.

The paper is organized as follows. In Section 2 we introduce some notions and auxiliary results needed for the solution of the problem. In Section 3 we prove some of the statements of Theorem \ref{theo:main:2} and establish the connection between the limiting functions $\widetilde{F}_{k}^{(\rho)}$ in \eqref{limitP}, used in \cite{LopLopstar} to express the ratio asymptotics of the MOP, and certain algebraic functions defined on $\mathcal{R}$. Section 4 is devoted to the proof of Theorem \ref{theo:main:1}, what remains of Theorem \ref{theo:main:2}, and the description of the functions $\widetilde{F}_{k}^{(\rho)}$.

\section {Auxiliary results}

We briefly recall some results from \cite{LopLopstar} that will be needed. As in \cite{LopMin,LopLopstar}, in this paper we will frequently use the notation
\[
[n:n']=\{s\in\mathbb{Z}: n\leq s\leq n'\},
\]
for any two integers $n\leq n'$. If $n'<n$, then $[n:n']$ indicates the empty set.

\subsection{Reduction to the real line}

For the study of multiple orthogonal polynomials on star-like sets, it is convenient to translate the problem to the real line.

Let
$(s_0,\ldots,s_{p-1}) = \mathcal{N}(\sigma_0,\ldots,\sigma_{p-1})$ be a Nikishin system on the star-like sets defined above. Along with the measures $s_j$ we also use the measures
\begin{equation}\label{def:skj}
 s_{k,j}=\langle \sigma_k,\ldots,\sigma_j\rangle,\quad 0\leq k\leq j\leq p-1.
\end{equation}
Notice that $(s_{k,k},\ldots,s_{k,j}) = \mathcal{N}(\sigma_k,\ldots,\sigma_j)$.

For every $0\leq j\leq p-1$, we shall denote by $\sigma^*_j$ the push-forward of $\sigma_j$ under the map $z\mapsto z^{p+1}$; that is,
 $\sigma^*_j$ is the measure on $[a_j,b_j]$ such that for every Borel set $E\subset [a_j,b_j]$,
\begin{equation}\label{def:sigma:star}
\sigma^*_j(E):=\sigma_j\left(\{z:z^{p+1}\in E\}\right).
\end{equation}
With the assumptions of Theorems \ref{theo:main:1}-\ref{theo:main:2}, it follows that $\sigma^*_j$ has positive Radon-Nikodym derivative a.e. with respect to Lebesgue measure on $\Delta_j$. Set
\[
\mu_{k,k}:=\sigma^*_k, \quad 0\leq k\leq p-1,
\]
\[
d
\mu_{k,j}(\tau):=\left(\tau\int_{a_{k+1}}^{b_{k+1}}\frac{d\mu_{k+1,j}(s)}{\tau-s}
\right)d\sigma^*_{k}(\tau),\quad \tau \in [a_k,b_k],\quad 0\leq k < j\leq p-1.
\]

The measures $s_{k,j}$ and $\mu_{k,j}$ are related through the formulas \cite[Prop. 2.2]{LopMin}
\begin{align*}
\int_{\Gamma_{k}}\frac{ds_{k,j}(t)}{z-t}={}&
z^{p+k-j}\int_{a_{k}}^{b_{k}}\frac{d\mu_{k,j}(\tau)}{z^{p+1}-\tau}, \qquad 0\leq k\leq j\leq p-1.
\end{align*}
That is,
\[
\widehat{s}_{k,j}(z)=z^{p+k-j}\widehat{\mu}_{k,j}(z^{p+1}).
\]

\subsection{Functions of the second kind}

For the asymptotic analysis of the multiple orthogonal polynomials, the functions of the second kind play a crucial role.

Observe that $\Psi_{n,k}$ is analytic in $\overline{\mathbb{C}}\setminus\Gamma_{k-1}$. It is not hard to deduce that for each fixed $k = 0,\ldots,p-1$, the function $\Psi_{n,k}$ satisfies orthogonality conditions with respect to the measures $s_{k,j}$, $j=k,\ldots,p-1$, defined in \eqref{def:skj}. This and other interesting properties of the functions $\Psi_{n,k}$ may be found in \cite[Propositions 2.5--2.7]{LopMin}. The functions $\Psi_{n,k}$ are linked with other functions of the second kind $\psi_{n,k}$ that can be defined in terms of the  push-forward measures introduced before on the real line.

\begin{definition}\label{definitionpsi}
Set $\psi_{n,0}:=\mathcal{Q}_d$, where $\mathcal{Q}_{d}$ is the polynomial that appears in the relation \eqref{eq:decompQn}, and for $1\leq k\leq p$, let
$\psi_{n,k}$ be the function analytic in $\mathbb{C}\setminus
[a_{k-1},b_{k-1}]$ defined as
\[
\psi_{n,k}(z):=\begin{cases}
z\int_{\Gamma_{k-1}}\frac{\Psi_{n,k-1}(t)\,t^{k-1-\ell}}{z-t^{p+1}}\,d\sigma_{
k-1}(t), & \ell< k,\\[1em]
\int_{\Gamma_{k-1}}\frac{\Psi_{n,k-1}(t)\,t^{p+k-\ell}}{z-t^{p+1}}\,d\sigma_{k-1
}(t), & k\leq \ell,
\end{cases}
\]
where $n\equiv \ell \mod (p+1)$, $0\leq \ell\leq p$.
\end{definition}

Let $n
\equiv \ell \mod (p+1)$ with $0\leq \ell\leq p$, and define
\begin{equation}\label{varymeas:sigma}
d\sigma_{n,k}(\tau):=\begin{cases}
d\sigma_{k}^{*}(\tau), & \ell\leq k,\\
\tau\,d\sigma_{k}^{*}(\tau), & k<\ell.
\end{cases}
\end{equation}Then,
\[
z^{k-\ell}\,\Psi_{n,k}(z)=\psi_{n,k}(z^{p+1}), \quad 0\leq k\leq p,
\]
and for all $1\leq k\leq p$,
\[
\psi_{n,k}(z)=\begin{cases}
z\int_{a_{k-1}}^{b_{k-1}}\frac{\psi_{n,k-1}(\tau)}{z-\tau}\,d\sigma_{n,k-1}
(\tau), & \ell< k,\\[1em]
\int_{a_{k-1}}^{b_{k-1}}\frac{\psi_{n,k-1}(\tau)}{z-\tau}\,d\sigma_{n,k-1}(\tau), & k\leq \ell.
\end{cases}
\]

The functions of the second kind satisfy the following recurrence relations. For every $n\geq p$, $0\leq k\leq p$, we have \cite[Proposition 3.2]{LopMin}
\[
z\Psi_{n,k}(z)=\Psi_{n+1,k}(z)+a_n\Psi_{n-p,k}(z),
\]
and if $n\equiv\ell\mod(p+1)$, $0\leq \ell \leq p-1$, then
\begin{equation}\label{threetermrecsecondkindpequena1}
\psi_{n,k}(z)=\psi_{n+1,k}(z)+a_n\psi_{n-p,k}(z),
\end{equation}
while if $n\equiv p\mod(p+1)$, then
\begin{equation}\label{threetermrecsecondkindpequena2}
z\psi_{n,k}(z)=\psi_{n+1,k}(z)+a_n\psi_{n-p,k}(z).
\end{equation}

For each $k=0,\ldots,p-1$, the function $\psi_{n,k}$ satisfies orthogonality conditions with respect to the measures $\mu_{k,j}, j=k\ldots,p-1$. We have  \cite[Proposition 2.10]{LopMin}:

Let $0\leq k\leq p-1$ and assume that $n\equiv \ell \mod (p+1)$ with $0\leq
\ell\leq p$. Then the function $\psi_{n,k}$ satisfies the orthogonality conditions
 \begin{equation}\label{orthogredPsink}
\int_{a_{k}}^{b_{k}}\psi_{n,k}(\tau)\,\tau^{s}\,d\mu_{k,j}(\tau)=0,\quad
\left\lceil\frac{\ell-j}{p+1}\right\rceil\leq s\leq \left\lfloor
\frac{n+p\ell-1-j(p+1)}{p(p+1)}\right\rfloor,\quad k\leq j\leq p-1.
\end{equation}

\subsection{Counting the number of orthogonality conditions}\label{counting}

For the asymptotic analysis of the multiple orthogonal polynomials and the functions of the second kind, it is crucial to have a control on the total number of orthogonality conditions in \eqref{orthogredPsink}. We define this quantity next
in the same way it was done in \cite{LopMin}.

\begin{definition}
Let  $n$ be a nonnegative integer and let $\ell=\ell(n)$ be the integer satisfying $n\equiv \ell\mod (p+1)$, $0\leq\ell\leq p$. For each
$0\leq j\leq p-1$, let $M_j=M_j(n)$ be the number of integers $s$
satisfying the inequalities
\begin{equation}\label{countingcond}
\left\lceil\frac{\ell-j}{p+1}\right\rceil\leq s\leq \left\lfloor
\frac{n+p\ell-1-j(p+1)}{p(p+1)}\right\rfloor.
\end{equation}
For each $0\leq k\leq p-1$, we define
\begin{equation}\label{def:Znk}
 Z(n,k):=\sum_{j=k}^{p-1}M_j.
\end{equation}
Also, by convention $Z(n,p):=0$.
\end{definition}

The importance of the quantities $Z(n,k)$ resides in the following results \cite[Proposition 2.19]{LopMin}, \cite[Theorem 3.1]{LopLopstar}:

For each $n\geq 0$ and $k=0,\ldots,p-1$, the function $\psi_{n,k}$ has exactly $Z(n,k)$
zeros in $\mathbb{C}\setminus([a_{k-1},b_{k-1}]\cup\{0\})$; they are all simple
and lie in the open interval $(a_{k},b_{k})$. The function $\psi_{n,p}$ has no
zeros in $\mathbb{C}\setminus([a_{p-1},b_{p-1}]\cup\{0\})$. The zeros of $\psi_{n+1,k}$ and $\psi_{n,k}$ on $(a_k,b_k)$ interlace.

In the study of ratio asymptotics in \cite{LopLopstar},  the quantities $Z(n+1,k)-Z(n,k)$ played a key role. In \cite[Lemma 4.1]{LopLopstar} it was proved that for each fixed $0\leq k\leq p-1$, the expression $Z(n+1,k)-Z(n,k)$ is periodic in $n$ with period $p(p+1)$, and $Z(n+1,k)-Z(n,k)\in\{-1,0,1\}$ for all $n$.

\subsection{The polynomials $P_{n,k}$}

\begin{definition}
For any integers $n\geq 0$ and $k$ with $0\leq k\leq p-1$, let $P_{n,k}$ denote the monic polynomial whose roots are the zeros of $\psi_{n,k}$ in
$(a_{k},b_{k})$. For convenience we also define the polynomials $P_{n,-1}\equiv
1$, $P_{n,p}\equiv 1$.
\end{definition}

According to what was said in the previous subsection about the zeros of $\psi_{n,k}$, we know that  $P_{n,k}$ has
degree $Z(n,k)$, all its zeros are simple, and interlace those of $P_{n+1,k}$. Recall that by Definition~\ref{definitionpsi}, $P_{n,0}=\psi_{n,0}$
is the polynomial $\mathcal{Q}_{d}$ that appears in \eqref{eq:decompQn} and, therefore,
\begin{equation}\label{eq:formZn0}
Z(n,0)=\deg(P_{n,0})=\left\lfloor\frac{n}{p+1}\right\rfloor.
\end{equation}

Taking into account \eqref{eq:decompQn}, the ratio asymptotics of the polynomials $Q_n$ reduces to that of the polynomials $P_{n,0}$. What is curious is that in order to solve this problem we need to study simultaneously the ratio asymptotics of all the sequences of polynomials $(P_{n,k}), n \geq 0,$ for $k=0.\ldots,p-1$.

The starting point of the present paper is the following result proved in \cite{LopLopstar}. Note that the condition stated in Proposition~\ref{prop:ratioP} for the measures $\sigma_{k}^{*}$ is equivalent to the condition required for the measures $\sigma_{k}$ on Theorems~\ref{theo:main:1}-\ref{theo:main:2}, but we prefer to state Proposition~\ref{prop:ratioP} as it is presented in \cite{LopLopstar}.

\begin{proposition}\label{prop:ratioP}
Assume that for each $k=0,\ldots,p-1$, the measure $\sigma_{k}^{*}$ defined in \eqref{def:sigma:star} has positive Radon-Nikodym derivative with respect to Lebesgue measure a.e. on $[a_{k},b_{k}]$. Let $0\leq \rho\leq p(p+1)-1$ be fixed. The following asymptotic properties hold:
\begin{itemize}
\item[1)] For each $k=0,\ldots,p-1,$
\begin{equation}\label{limitP}
\lim_{\lambda \to \infty}\frac{P_{\lambda p(p+1)+\rho+1,k}(z)}{P_{\lambda p(p+1)+\rho,k}(z)}=\widetilde{F}_{k}^{(\rho)}(z),\qquad z\in\mathbb{C}\setminus[a_{k},b_{k}],
\end{equation}
where  $\widetilde{F}_k^{(\rho)}$ is holomorphic in $\mathbb{C}\setminus[a_{k},b_{k}]$.
\item[2)] If $\rho\not\equiv p \mod (p+1)$, then
\begin{equation}\label{eq:ratioQn:1}
\lim_{\lambda\rightarrow\infty}\frac{Q_{\lambda p(p+1)+\rho+1}(z)}{Q_{\lambda p(p+1)+\rho}(z)}=z\,\widetilde{F}_{0}^{(\rho)}(z^{p+1}),\qquad z\in\mathbb{C}\setminus\Gamma_{0},
\end{equation}
and if $\rho\equiv p\mod (p+1)$, then
\begin{equation}\label{eq:ratioQn:2}
\lim_{\lambda\rightarrow\infty}\frac{Q_{\lambda p(p+1)+\rho+1}(z)}{Q_{\lambda p(p+1)+\rho}(z)}=\frac{\widetilde{F}_{0}^{(\rho)}(z^{p+1})}{z^{p}},\qquad z\in\mathbb{C}\setminus(\Gamma_{0}\cup\{0\}).
\end{equation}
\item[3)] The sequence $(a_{n})$ of recurrence coefficients in \eqref{threetermrec} satisfies
\begin{equation}\label{eq:limreccoeff}
\lim_{\lambda\rightarrow\infty}a_{\lambda p(p+1)+\rho}=a^{(\rho)},
\end{equation}
where the limiting values $a^{(\rho)}$ appear in the Laurent expansion at infinity of $\widetilde{F}_{0}^{(\rho)}$ as follows:
\begin{equation}\label{eq:LaurentFexp}
\widetilde{F}_{0}^{(\rho)}(z)=\begin{cases}
1-a^{(\rho)} z^{-1}+O\left(z^{-2}\right), & \ \mbox{if} \ \rho\not\equiv p \mod (p+1),\\
z-a^{(\rho)}+O\left(z^{-1}\right), & \ \mbox{if} \ \rho\equiv p \mod (p+1).
\end{cases}
\end{equation}
\end{itemize}
\end{proposition}

In \cite[Section~6.3]{LopLopstar} it was proved that for each $0\leq k\leq p-1$, $0\leq \rho\leq p(p+1)-1$, the function $\widetilde{F}_k^{(\rho)}$ is either a Szeg\H{o} function in $\overline{\mathbb{C}}\setminus [a_{k},b_{k}]$, or it is the product or division of such a function by a conformal map of $\overline{\mathbb{C}} \setminus[a_{k},b_{k}]$ onto the exterior of the unit disk (for more details, see the second paragraph in the proof of Theorem~\ref{lem:descfkrho} below). From this it was shown that the limit
\begin{equation}\label{absolute}
|\widetilde{F}_k^{(\rho)}(x)| :=\lim_{z\rightarrow x}|\widetilde{F}_k^{(\rho)}(z)|, \qquad x \in [a_{k},b_{k}] \setminus \{0\},
\end{equation}
exists for all the points $x$ specified, as $z\in\mathbb{C}\setminus[a_{k},b_{k}]$ approaches $x$. If $0\in[a_{k},b_{k}]$, then \eqref{absolute} also holds for $x=0$ provided that the weight associated with the Szeg\H{o} function is positive and continuous at that point.

In \cite[Lemma 6.4]{LopLopstar} it was shown that a normalization of the functions $(\widetilde{F}_k^{(\rho)})_{k=0}^{p-1}$ constitute the solution of a system of boundary value equations which we restate here for convenience of the reader.

\begin{proposition}\label{prop:boundary}
Let $\rho\in[0:p(p+1)-1]$ be fixed, and let $\ell\in[0:p]$ be the remainder in the division of $\rho$ by $p+1$. Let $\Fkr$, $0\leq k\leq p-1$, be the limiting functions in \eqref{limitP}. Then there exist positive constants $c_{k}^{(\rho)}$ so that the collection of functions $F_{k}^{(\rho)}=c_{k}^{(\rho)} \Fkr, 0\leq k \leq p-1,$ or more precisely their absolute values (taking account of \eqref{absolute}),  satisfies the following systems of boundary value equations:
\begin{itemize}
\item[1)] When $\ell \in [0:p-1]$ (here, $[0:-1], [p:p-1]$ and $[p+1:p-1]$ denote the empty set for the corresponding values of $\ell$)
\begin{align}
\frac{|\Fk(\tau)|^{2}}{|\Fkm(\tau)||\Fkp(\tau)|}= & 1, \quad \tau\in[a_{k},b_{k}], \quad k \in [0:\ell -1] \cup [\ell +2:p-1], \label{eq:bv:lnp:1}\\
\frac{|\Fk(\tau)|^{2}\,|\tau|}{|\Fkm(\tau)||\Fkp(\tau)|}= & 1, \quad \tau\in[a_{k},b_{k}] \setminus \{0\}, \quad k = \ell, \label{eq:bv:lnp:2}\\
\frac{|\Fk(\tau)|^{2}}{|\tau||\Fkm(\tau)||\Fkp(\tau)|}= & 1, \quad \tau\in[a_{k},b_{k}] \setminus \{0\}, \quad k = \ell +1. \label{eq:bv:lnp:3}
\end{align}
(The last equation is dropped if $\ell = p-1$.)
\item[2)] For $\ell=p$, the system is
\begin{align}
\frac{|F_{0}^{(\rho)}(\tau)|^{2}}{|\tau||F_{1}^{(\rho)}(\tau)|} = & 1, \quad \tau\in [a_{0},b_{0}] \setminus \{0\},\label{eq:bv:lp:1}\\
\frac{|\Fk(\tau)|^{2}}{|\Fkm(\tau)||\Fkp(\tau)|} = & 1, \quad \tau\in[a_{k},b_{k}],\quad k\in [1:p-1].\label{eq:bv:lp:2}
\end{align}
In the above equations \eqref{eq:bv:lnp:1}--\eqref{eq:bv:lp:2}, we use the convention $F_{-1}^{(\rho)}\equiv F_{p}^{(\rho)}\equiv 1$.
\end{itemize}
Moreover, for each $\rho$ fixed, the functions $F_{k}^{(\rho)}(z), 0\leq k \leq p-1$ satisfy:
\begin{itemize}
\item[i)] $(F_k^{(\rho)})^{\pm 1} \in \mathcal{H}(\mathbb{C}\setminus [a_k,b_k])$ (holomorphic in the specified domain).
\item[ii)] The leading coefficient (corresponding to the highest power of $z$) of the Laurent expansion of $F_k^{(\rho)}$  at $\infty$ is positive.
\item[iii)] $F_k^{(\rho)}$ either has a simple pole, a simple zero, or takes a finite positive value at $\infty$. For a given $\rho \in [0:p(p+1)-1]$ and $k \in [0:p-1]$, only one of these situations occurs.
 \end{itemize}
\end{proposition}

In this paper, we express the functions $(\widetilde{F}_k^{(\rho)})_{k=0}^{p-1}$ in terms of the conformal mappings on the Riemann surface $\mathcal{R}$ defined in the introduction, see Theorem~\ref{theo:descFkrconf}.

Throughout the rest of the paper, we extend the $p$ sequences $\{\widetilde{F}_{k}^{(\rho)}\}_{\rho=0}^{p(p+1)-1}$, $0\leq k\leq p-1$, and their normalizations, periodically with period $p(p+1)$, to allow the super-index $\rho$ to take arbitrary integer values. So by definition we set
\[
\widetilde{F}^{(\rho+p(p+1))}_{k}\equiv \widetilde{F}^{(\rho)}_{k},\qquad \rho\in\mathbb{Z},\quad k\in[0:p-1].
\]
Recall that we also extend the sequence $\{a^{(\rho)}\}_{\rho=0}^{p(p+1)-1}$
of limiting values in \eqref{eq:limreccoeff}, periodically with period $p(p+1)$, so that
\[
a^{(\rho+p(p+1))}=a^{(\rho)},\qquad \rho\in\mathbb{Z}.
\]

\section{The boundary value problem and algebraic functions}

\subsection{Some additional notation}
Let us define $\omega_{l,j}$ as the leading coefficient in the Laurent series expansion of $\varphi_{j}^{(l)}$ at $\infty$, i.e.,
\begin{equation}\label{def:omegalj}
\omega_{l,j}:=\begin{cases}
\omega_{l} & j=0,\\[0.2em]
(\varphi_{l}^{(l)})'(\infty) & j=l,\\[0.2em]
\varphi_{j}^{(l)}(\infty) & 1\leq j\leq p,\,\,j\neq l.
\end{cases}
\end{equation}
As was proved in the Introduction, the condition $\omega_{l}>0$ implies that
\begin{equation}\label{symmRS}
\varphi_k^{(l)}(z)=\overline{\varphi_k^{(l)}(\overline{z})}, \qquad k=0,\ldots,p.
\end{equation}

Throughout the rest of the paper, we use the following notation, already employed for the functions $\widetilde{F}_{k}^{(\rho)}$. Given an arbitrary function $F(z)$ which has in a neighborhood of infinity a Laurent expansion of the form $F(z)= C z^{k}+O(z^{k-1})$, $C\neq 0$, $k\in\mathbb{Z}$, we denote $\widetilde{F}(z):=F(z)/C$. If $C$ is real, $\mathrm{sg}(F(\infty))$ will represent the sign of $C$.

The symmetry property \eqref{symmRS} implies that for each $k=0,\ldots,p,$ the function $\varphi_{k}^{(l)}$ is real-valued on $\mathbb{R}\setminus(\Delta_{k-1}\cup\Delta_{k})$, where $\Delta_{-1}=\Delta_{p}=\emptyset$. This, and the fact that $\varphi^{(l)}:\mathcal{R}\longrightarrow\overline{\mathbb{C}}$ is a bijection, easily imply the following statements, which are left to the reader to check: If $l\in\{1,\ldots,p\}$ is odd, then
\begin{equation}\label{eq:signphikl:1}
\mathrm{sg}(\varphi_{k}^{(l)}(\infty))=\begin{cases}
+1 & \mbox{for}\quad 0\leq k\leq l,\\
-1 & \mbox{for}\,\,l<k\leq p,
\end{cases}
\end{equation}
and if $l\in\{1,\ldots,p\}$ is even, then
\begin{equation}\label{eq:signphikl:2}
\mathrm{sg}(\varphi_{k}^{(l)}(\infty))=\begin{cases}
+1 & \mbox{for}\,\,0\leq k< l,\\
-1 & \mbox{for}\,\,l\leq k\leq p.
\end{cases}
\end{equation}

\subsection{Proof of $1)$ in  Theorem \ref{theo:main:2}}
\begin{proof}
We aim to prove that for each $0\leq \rho\leq p(p+1)-1$, the limiting value on the right-hand side of \eqref{eq:limreccoeff} satisfies $a^{(\rho)}>0$.

The following relations were proved in \cite{LopLopstar}, and are easily obtained applying \eqref{limitP} and \eqref{threetermrecsecondkindpequena1}--\eqref{threetermrecsecondkindpequena2} in the case $k=0$. We have
\begin{align}
a^{(\rho)} & =(z-\widetilde{F}_{0}^{(\rho)}(z)) \prod_{i=\rho-p}^{\rho-1}\widetilde{F}_{0}^{(i)}(z),\qquad z\in\mathbb{C}\setminus[a_{0},b_{0}],\quad \rho\equiv p \mod (p+1),\label{eq:relarhoF:1}\\
a^{(\rho)} & =(1-\widetilde{F}_{0}^{(\rho)}(z)) \prod_{i=\rho-p}^{\rho-1}\widetilde{F}_{0}^{(i)}(z),\qquad z\in\mathbb{C}\setminus[a_{0},b_{0}],\quad \rho\not\equiv p \mod (p+1).\label{eq:relarhoF:2}
\end{align}

Assume that $a^{(\rho)}=0$ for some $\rho$ satisfying $\rho\equiv p \mod (p+1)$. Since none of the functions $\widetilde{F}_{0}^{(i)}$ vanish on $\mathbb{C}\setminus[a_{0}, b_{0}]$ (cf. Proposition~\ref{prop:boundary}~i)), we deduce from \eqref{eq:relarhoF:1} that $\widetilde{F}_{0}^{(\rho)}(z)\equiv z$ on that domain. If $0\notin[a_{0},b_{0}]$, then $\widetilde{F}_{0}^{(\rho)}(z)$ has a zero at the origin, which contradicts Proposition~\ref{prop:boundary}~i). Suppose that $a_{0}=0$. Then, $F_{0}^{(\rho)}(z)=c_{0}^{(\rho)}\widetilde{F}_{0}^{(\rho)}(z)=c_{0}^{(\rho)} z$ and \eqref{eq:bv:lp:1} imply that $F_{1}^{(\rho)}(z)=(c_{0}^{(\rho)})^2\,z$, which contradicts the fact that $F_{1}^{(\rho)}(z)$ does not vanish in the exterior of $[a_{1},b_{1}]$, which is disjoint from $[a_{0},b_{0}]=[0,b_{0}]$.

Now assume that $a^{(\rho)}=0$ for some $\rho\not\equiv p \mod (p+1)$. Then, from \eqref{eq:relarhoF:2} we deduce that $\widetilde{F}_{0}^{(\rho)}(z)\equiv 1$ on $\mathbb{C}\setminus[a_{0},b_{0}]$.

Suppose first that $\ell=0$ (see the statement of Proposition~\ref{prop:boundary} for the definition of $\ell$). Applying \eqref{eq:bv:lnp:2} for $k=0$ we get $F_{1}^{(\rho)}(z)=(c_{0}^{(\rho)})^{2}\, z$. If $0\not\in[a_{1},b_{1}]$, then $F_{1}^{(\rho)}$ has a zero outside $[a_{1},b_{1}]$, which is contradictory with the non-vanishing property. Now assume $0\in[a_{1}, b_{1}]$, i.e., $b_{1}=0$. If $p\geq 3$, then applying \eqref{eq:bv:lnp:3} for $k=1$ we obtain that the function $F_{2}^{(\rho)}$ must have a zero at the origin, contradiction. If $p=2$, then $F_{2}^{(\rho)}\equiv 1$ by definition, and \eqref{eq:bv:lnp:3} reduces to $c |\tau|=1$, $\tau\in[a_{1},0)$, $c$ a constant, which is impossible.

Now suppose that $1\leq \ell\leq p-2$. Applying \eqref{eq:bv:lnp:1} repeatedly for $k=0,\ldots,\ell-1,$ we obtain that the functions $F_{k}^{(\rho)}$, $0\leq k\leq \ell$, are all constant in their domains. Then from equation \eqref{eq:bv:lnp:2} we deduce that $F_{\ell+1}^{(\rho)}(z)=c\,z$ for some constant $c$. If $0\not\in[a_{\ell+1},b_{\ell+1}]$, contradiction. So assume that $0\in[a_{\ell+1},b_{\ell+1}]$. From \eqref{eq:bv:lnp:3} we now obtain that $F_{\ell+2}^{(\rho)}$ must have a zero at $0\in\mathbb{C}\setminus[a_{\ell+2},b_{\ell+2}]$, which is a contradiction.

Finally, assume that $\ell=p-1$ (we also assume that $p\geq 2$). Recall that by assumption $\widetilde{F}_{0}^{(\rho)}\equiv 1$. Applying \eqref{eq:bv:lnp:1} repeatedly for $k=0,\ldots,\ell-1,$ we obtain that the functions $F_{k}^{(\rho)}$, $0\leq k\leq p-1$, are all constant in their domains. This contradicts \eqref{eq:bv:lnp:2}, since $F_{p}^{(\rho)}\equiv 1$.
\end{proof}

\subsection{A fundamental relation}

We wish to express the functions that solve the system of boundary value equations in Proposition \ref{prop:boundary} in terms of algebraic functions defined on the Riemann surface. A direct relation is hard to establish, but if one multiplies $p+1$ consecutive $F_k^{(\rho)}$ as it is done in \eqref{def:fkrho}, then such a product has a very nice representation (see \eqref{eq:expfkrho} below). In order to arrive to that formula we need to analyze the order of such products at infinity. For this purpose we introduce the following quantities.

For integers $n\geq 0$ and $k\in[0: p]$, let
\begin{equation}\label{def:Lambdank}
\Lambda(n,k):=Z(n+p+1,k)-Z(n,k).
\end{equation}
Note also that
\begin{equation}\label{eq:altformL}
\Lambda(n,k)=\sum_{j=0}^{p}\left(Z(n+j+1,k)-Z(n+j,k)\right),
\end{equation}
which will be used later.

\begin{lemma}
For any integers $n\geq 0$ and $k\in[0: p]$,
\begin{equation}\label{eq:descLambda}
\Lambda(n,k)=\begin{cases}
0, & \mbox{if}\,\,\,\,n\equiv s\mod p,\,\,\,\, s\in[0:k-1],\\
1, & \mbox{if}\,\,\,\,n\equiv s\mod p,\,\,\,\, s\in[k:p-1].
\end{cases}
\end{equation}
In particular, for each $k\in[0: p]$, $\Lambda(n,k)$ is periodic as a function of $n$ with period $p$.
\end{lemma}
\begin{proof}
For an integer $n\geq 0$, let $\ell = \ell(n)$ be the integer satisfying $n\equiv \ell\mod (p+1)$, $0\leq \ell\leq p$. According to \eqref{countingcond} and \eqref{def:Znk},
\[
Z(n,k)=\sum_{j=k}^{p-1}M_{j}(n),
\]
where
\[
M_{j}(n)=\left\lfloor\frac{n+p\,\ell(n)-1-j (p+1)}{p(p+1)}\right\rfloor-\left\lceil\frac{\ell(n)-j}{p+1}\right\rceil+1.
\]
Since $\ell(n)=\ell(n+p+1)=\ell$, we obtain
\begin{align*}
\Lambda(n,k) & =\sum_{j=k}^{p-1}\left(M_{j}(n+p+1)-M_{j}(n)\right)\\
& =\sum_{j=k}^{p-1}\left(\left\lfloor\frac{n+p+1+p\,\ell-1-j (p+1)}{p(p+1)}\right\rfloor-\left\lfloor\frac{n+p\,\ell-1-j (p+1)}{p(p+1)}\right\rfloor\right)\\
& =\sum_{j=k}^{p-1}\left(\left\lfloor\frac{n+p\,\ell-1-j (p+1)}{p(p+1)}+\frac{1}{p}\right\rfloor-\left\lfloor\frac{n+p\,\ell-1-j (p+1)}{p(p+1)}\right\rfloor\right).
\end{align*}
We can write
\[
n=\lambda p(p+1)+\rho,\qquad \lambda\geq 0,\quad 0\leq \rho\leq p(p+1)-1,
\]
and decompose $\rho$ as
\[
\rho=\eta (p+1)+\ell,\qquad 0\leq \eta\leq p-1.
\]
Then
\[
\frac{n+p\, \ell-1-j(p+1)}{p(p+1)}=\lambda+\frac{\eta+\ell-j}{p}-\frac{1}{p(p+1)}.
\]
Let $s\in[0:p-1]$ be the residue of $n$ modulo $p$. Note that $n\equiv (\eta+\ell) \mod p$, so if we write $\eta+\ell=s+mp$ for some integer $m$, we get
\begin{equation}\label{eq:decomp:1}
\frac{n+p\, \ell-1-j(p+1)}{p(p+1)}=\lambda+m+\frac{s-j}{p}-\frac{1}{p(p+1)}.
\end{equation}

Assume first that $s\in[0:k-1]$. Then, from \eqref{eq:decomp:1} we obtain that for every $j\in[k:p-1]$,
\[
\left\lfloor\frac{n+p\, \ell-1-j(p+1)}{p(p+1)}\right\rfloor=\left\lfloor\frac{n+p \ell-1-j(p+1)}{p(p+1)}+\frac{1}{p}\right\rfloor=\lambda+m-1,
\]
which implies that $\Lambda(n,k)=0$. If $s\in[k:p-1]$, then
\[
\left\lfloor\frac{n+p\, \ell-1-j(p+1)}{p(p+1)}\right\rfloor=\begin{cases}
\lambda+m & \mbox{if}\quad k\leq j\leq s-1,\\
\lambda+m-1 & \mbox{if}\quad s\leq j\leq p-1,
\end{cases}
\]
and
\[
\left\lfloor\frac{n+p\, \ell-1-j(p+1)}{p(p+1)}+\frac{1}{p}\right\rfloor=\begin{cases}
\lambda+m & \mbox{if}\quad k\leq j\leq s,\\
\lambda+m-1 & \mbox{if}\quad s+1\leq j\leq p-1,
\end{cases}
\]
which implies that in this case $\Lambda(n,k)=1$.
\end{proof}

\begin{definition}
For each $k\in[0:p-1]$ and $\rho\in\mathbb{Z}$, we define
\begin{equation}\label{def:fkrho}
f_{k}^{(\rho)}(z):=\prod_{j=0}^{p} F_{k}^{(\rho+j)}(z), \qquad z \in \mathbb{C} \setminus [a_k,b_k].
\end{equation}
We also set $f_{-1}^{(\rho)}\equiv f_{p}^{(\rho)}\equiv 1$.

\end{definition}

\begin{theorem}\label{lem:descfkrho}
The functions defined in \eqref{def:fkrho} satisfy the following properties for each $k\in[0:p-1]$ and $\rho\in[0:p(p+1)-1]$:
\begin{itemize}
\item[1)] $(f_{k}^{(\rho)})^{\pm 1}\in\mathcal{H}(\mathbb{C}\setminus\Delta_{k})$, and as $z\rightarrow\infty$,
\begin{equation}\label{eq:estfkrhoinf}
f_{k}^{(\rho)}(z)=c_{k,\rho}\,z^{\Lambda(\rho,k)}(1+O(z^{-1})),
\end{equation}
where $c_{k,\rho}$ is a positive constant, and $\Lambda(\rho,k)$ is described in \eqref{eq:descLambda}.
\item[2)] The function $|f_{k}^{(\rho)}|$ has continuous and strictly positive boundary values on all $\Delta_{k}$ and we have
\begin{equation}\label{eq:boundvalfkr}
\frac{|f_{k}^{(\rho)}(\tau)|^{2}}{|f_{k-1}^{(\rho)}(\tau)||f_{k+1}^{(\rho)}(\tau)|}=1,\qquad \tau\in\Delta_{k}.
\end{equation}
\item[3)] Let $l=l(\rho)$ be the integer determined by the conditions $l-1\equiv \rho \mod p$ and $1\leq l\leq p$. Then
\begin{equation}\label{eq:expfkrho}
f_{k}^{(\rho)}(z)=\mathrm{sg}\left(\prod_{\nu=k+1}^{p}\varphi_{\nu}^{(l)}(\infty)\right)\prod_{\nu=k+1}^{p}\varphi_{\nu}^{(l)}(z),\qquad z\in\overline{\mathbb{C}}\setminus\Delta_{k},
\end{equation}
where
\begin{equation}\label{eq:signfkrho}
\mathrm{sg}\left(\prod_{\nu=k+1}^{p}\varphi_{\nu}^{(l)}(\infty)\right)=\begin{cases}
(-1)^{p+1} & \mbox{if}\quad 0\leq k\leq l-1,\\[0.5em]
(-1)^{p+k} & \mbox{if}\quad l\leq k\leq p-1.
\end{cases}
\end{equation}
\end{itemize}
\end{theorem}
\begin{proof}
Recall that the polynomial $P_{n,k}$ has degree $Z(n,k)$. Therefore, from \eqref{limitP} and the fact that $Z(n+1,k)-Z(n,k)$ is periodic with respect to $n$ with period $p(p+1)$, it follows that
\begin{equation}\label{eq:estFkrhoinf}
\widetilde{F}_{k}^{(\rho+j)}(z)=z^{Z(\rho+j+1,k)-Z(\rho+j,k)}(1+O(z^{-1})),\qquad z\rightarrow\infty,\quad 0\leq j\leq p.
\end{equation}
Since $F_{k}^{(\rho+j)}=\widetilde{c}_{k,j} \widetilde{F}_{k}^{(\rho+j)}(z)$, where $\widetilde{c}_{k,j}$ is a positive constant (cf. Proposition~\ref{prop:boundary}), multiplying the $p+1$ estimates in \eqref{eq:estFkrhoinf} and applying \eqref{eq:altformL} and \eqref{def:fkrho}, we obtain \eqref{eq:estfkrhoinf}. We have $(f_{k}^{(\rho)})^{\pm 1}\in\mathcal{H}(\mathbb{C}\setminus\Delta_{k})$ since none of the functions $F_{k}^{(\rho)}$ vanish on $\mathbb{C}\setminus\Delta_{k}$.

In \cite[Section~6.3]{LopLopstar} it was shown the following. Up to a multiplicative constant, each function $\widetilde{F}_{k}^{(\rho+j)}$, $0\leq j\leq p$, can be expressed either as a Szeg\H{o} function, or as a Szeg\H{o} function multiplied or divided by the conformal mapping $\phi_{k}$ from the exterior of $\Delta_{k}$ onto the exterior of the unit circle that satisfies $\phi_{k}(\infty)=\infty$ and $\phi_{k}'(\infty)>0$. The Szeg\H{o} function in the expression of $\widetilde{F}_{k}^{(\rho+j)}$ is associated with a weight that takes one of the following three forms:
\begin{equation}\label{eq:Szegoweights}
\frac{1}{|\widetilde{F}_{k-1}^{(\rho+j)} (\tau)||\widetilde{F}_{k+1}^{(\rho+j)} (\tau)|}, \qquad \frac{|\tau|}{|\widetilde{F}_{k-1}^{(\rho+j)} (\tau)||\widetilde{F}_{k+1}^{(\rho+j)} (\tau)|},\qquad \frac{1}{|\tau||\widetilde{F}_{k-1}^{(\rho+j)} (\tau)||\widetilde{F}_{k+1}^{(\rho+j)}(\tau)|}.\end{equation}
A careful analysis of the different cases described in \cite[Section~6.3]{LopLopstar}, shows that as $j$ varies in the range $[0:p]$, exactly one $j$ corresponds to a weight of the second type (the $j$ satisfying $\rho+j \equiv k \mod (p+1)$), exactly one $j$ corresponds to a weight of the third type (the $j$ satisfying $\rho+j \equiv (k-1) \mod (p+1)$), and all other $j$ correspond to a weight of the first type. By the multiplicative property of Szeg\H{o} functions, the possible singularities that $|\tau|$ and $1/|\tau|$ in \eqref{eq:Szegoweights} may cause at the origin will not be present in the product $f_{k}^{(\rho)}$. Hence, $|f_{k}^{(\rho)}|$ will have continuous and non-vanishing boundary values on all $\Delta_{k}$. Multiplying the different boundary value equations in Proposition~\ref{prop:boundary} for the different indices $\rho+j$, $0\leq j\leq p$, we obtain \eqref{eq:boundvalfkr} (the reader can also observe the cancellation between $|\tau|$ and $1/|\tau|$ after multiplying these equations).

Let $\rho\in[0:p(p+1)-1]$ be fixed, and let $l$ be the integer satisfying $l-1\equiv \rho \mod p$, $1\leq l\leq p$. Then we have shown that the system of functions $\{f_{k}^{(\rho)}\}_{k=0}^{p-1}$ satisfies the following conditions:
\begin{itemize}
\item[a)] $f_{k}^{(\rho)}, 1/f_{k}^{(\rho)}\in\mathcal{H}(\mathbb{C}\setminus\Delta_{k})$, $k=0,\ldots,p-1$.
\item[b)] In virtue of \eqref{eq:descLambda} and \eqref{eq:estfkrhoinf}, as $z\rightarrow\infty$ we have
\[
f_{k}^{(\rho)}(z)=\begin{cases}
c_{k,\rho} z+O(1),\quad 0\leq k\leq l-1,\\
c_{k,\rho}+O(z^{-1}),\quad l\leq k\leq p-1,
\end{cases}
\]
where $c_{k,\rho}>0$ for all $0\leq k\leq p-1$.
\item[c)] The boundary value relation \eqref{eq:boundvalfkr} holds for each $0\leq k\leq p-1$.
\end{itemize}

In \cite[Lemma 4.2]{AptLopRocha} it was proved that the boundary value problem a)-b)-c) has a unique solution and it is precisely given by
\[
f_{k}^{(\rho)}(z)=\mathrm{sg}\left(\prod_{\nu=k+1}^{p}\varphi_{\nu}^{(l)}(\infty)\right)\prod_{\nu=k+1}^{p}\varphi_{\nu}^{(l)}(z),\qquad z\in\overline{\mathbb{C}}\setminus\Delta_{k}.
\]
Formula \eqref{eq:signfkrho} follows immediately from \eqref{eq:signphikl:1} and \eqref{eq:signphikl:2}.
\end{proof}

\begin{corollary}
The following properties hold:
\begin{itemize}
\item[1)] For each $\rho\in[0:p(p+1)-1]$ and $k\in[0:p-1]$,
\begin{equation}\label{eq:perfkrho}
f_{k}^{(\rho)}\equiv f_{k}^{(\rho+p)}.
\end{equation}
\item[2)] For each $\rho\in[0:p(p+1)-1]$ and $k\in[0:p-1]$,
\begin{equation}\label{eq:idprodFkrho}
\prod_{i=\rho}^{\rho+p-1} F_{k}^{(i)}\equiv \prod_{i=\rho+p+1}^{\rho+2p} F_{k}^{(i)},\qquad \prod_{i=\rho}^{\rho+p-1} \widetilde{F}_{k}^{(i)}\equiv \prod_{i=\rho+p+1}^{\rho+2p} \widetilde{F}_{k}^{(i)}.
\end{equation}
\item[3)] For each $\rho\in[0:p(p+1)-1]$,
\begin{equation}\label{eq:idsumakrho}
\sum_{i=\rho}^{\rho+p-1}a^{(i)}=\sum_{i=\rho+p+1}^{\rho+2p} a^{(i)}.
\end{equation}
\item[4)] For each $\rho\in[0:p(p+1)-1]$ and $z\in\mathbb{C}\setminus\Delta_{0}$ we have
\begin{align}
\frac{a^{(\rho+p+1)}}{a^{(\rho)}} & =\frac{z-\widetilde{F}_{0}^{(\rho+p+1)}(z)}{z-\widetilde{F}_{0}^{(\rho)}(z)}\qquad\mbox{if}\,\,\,\,\rho\equiv p \mod (p+1),\label{eq:idquotFkrho:1}\\
\frac{a^{(\rho+p+1)}}{a^{(\rho)}} & =\frac{1-\widetilde{F}_{0}^{(\rho+p+1)}(z)}{1-\widetilde{F}_{0}^{(\rho)}(z)}\qquad\mbox{if}\,\,\,\,\rho\not\equiv p \mod (p+1).\label{eq:idquotFkrho:2}
\end{align}
\end{itemize}
\end{corollary}
\begin{proof}
The relation \eqref{eq:perfkrho} follows immediately from \eqref{eq:expfkrho} since $l(\rho)=l(\rho+p)$, and \eqref{eq:idprodFkrho} is obtained dividing both sides of \eqref{eq:perfkrho} by $F_{k}^{(\rho+p)}$.

Taking $k=0$ in \eqref{eq:idprodFkrho} we get
\begin{equation}\label{eq:relnormFkrho}
\prod_{i=\rho}^{\rho+p-1}\widetilde{F}_{0}^{(i)}\equiv \prod_{i=\rho+p+1}^{\rho+2p} \widetilde{F}_{0}^{(i)}.
\end{equation}
In virtue of \eqref{eq:LaurentFexp}, as $z\rightarrow\infty$ we have
\begin{equation}\label{eq:relexpFkrho}
\prod_{i=\rho}^{\rho+p-1} \widetilde{F}_{0}^{(i)}(z)=
\begin{cases}
1-\left(\sum_{i=\rho}^{\rho+p-1}a^{(i)}\right) z^{-1}+O(z^{-2}),\quad \rho\equiv 0 \mod (p+1),\\[1em]
z-\sum_{i=\rho}^{\rho+p-1}a^{(i)}+O(z^{-1}),\quad \rho\not\equiv 0 \mod (p+1),
\end{cases}
\end{equation}
hence \eqref{eq:idsumakrho} is a consequence of \eqref{eq:relnormFkrho} and \eqref{eq:relexpFkrho}. Notice that \eqref{eq:idsumakrho} is the statement $3)$ of Theorem \ref{theo:main:2}.

Assume that $\rho\equiv p \mod (p+1)$. According to \eqref{eq:relarhoF:1}, we have
\begin{align*}
a^{(\rho)} & =(z-\widetilde{F}_{0}^{(\rho)}(z)) \prod_{i=\rho-p}^{\rho-1} \widetilde{F}_{0}^{(i)}(z),\qquad z\in\mathbb{C}\setminus\Delta_{0},\\
a^{(\rho+p+1)} & =(z-\widetilde{F}_{0}^{(\rho+p+1)}(z)) \prod_{i=\rho+1}^{\rho+p} \widetilde{F}_{0}^{(i)}(z),\qquad z\in\mathbb{C}\setminus\Delta_{0}.
\end{align*}
Dividing the second identity by the first identity, and applying Theorem~\ref{theo:main:2}.1 and \eqref{eq:idprodFkrho}, we obtain \eqref{eq:idquotFkrho:1}. Similarly one proves \eqref{eq:idquotFkrho:2}, using \eqref{eq:relarhoF:2}.\end{proof}

\subsection{Proof of $2)$ in Theorem \ref{theo:main:2}}

\begin{proof}
First note that if $l_{1}, l_{2}\in[1:p]$ with $l_{1}\neq l_{2}$, then $\varphi^{(l_{1})}/\varphi^{(l_2)}:\mathcal{R}\longrightarrow\overline{\mathbb{C}}$ is conformal. Indeed, from the definition of $\varphi^{(l)}$ we deduce that $\varphi^{(l_{1})}/\varphi^{(l_{2})}$ is a meromorphic function on $\mathcal{R}$ with only one simple pole (the point $\infty^{(l_{1})}$) and only one simple zero (the point $\infty^{(l_{2})}$).

Let $m_1$, $m_2$ be indices such that $0\leq m_1<m_2\leq p-1$. In virtue of \eqref{eq:idquotFkrho:1}--\eqref{eq:idquotFkrho:2} we have
\[
\frac{a^{(\rho+m_2(p+1))}}{a^{(\rho+m_1(p+1))}}=\begin{cases}
\frac{z-\widetilde{F}_{0}^{(\rho+m_2(p+1))}(z)}{z-\widetilde{F}_{0}^{(\rho+m_1(p+1))}(z)}\quad\mbox{if}\,\,\rho\equiv p \mod (p+1),\\[1em]
\frac{1-\widetilde{F}_{0}^{(\rho+m_2(p+1))}(z)}{1-\widetilde{F}_{0}^{(\rho+m_1(p+1))}(z)}\quad\mbox{if}\,\,\rho\not\equiv p \mod (p+1).
\end{cases}
\]
Let us assume that $a^{(\rho+m_1(p+1))}=a^{(\rho+m_2(p+1))}$. Then, from the above relation we deduce that $\widetilde{F}_{0}^{(\rho+m_1(p+1))}\equiv \widetilde{F}_{0}^{(\rho+m_2(p+1))}$.

In virtue of \eqref{def:fkrho}, we have
\[
\frac{\widetilde{f}_{0}^{(\rho+1)}}{\widetilde{f}_{0}^{(\rho)}}=\frac{\widetilde{F}_{0}^{(\rho+p+1)}}{\widetilde{F}_{0}^{(\rho)}},\quad \mbox{for any}\,\,\rho\in\mathbb{Z},
\]
so a repeated application of this identity yields
\begin{equation}\label{eq:auxquotF}
\frac{\widetilde{F}_{0}^{(\rho+m_2(p+1))}}{\widetilde{F}_{0}^{(\rho+m_1(p+1))}}=\prod_{m=m_1}^{m_2-1}\frac{\widetilde{F}_{0}^{(\rho+(m+1)(p+1))}}{\widetilde{F}_{0}^{(\rho+m(p+1))}}=\prod_{m=m_1}^{m_2-1}\frac{\widetilde{f}_{0}^{(\rho+m(p+1)+1)}}{\widetilde{f}_{0}^{(\rho+m(p+1))}}.
\end{equation}
On the other hand, by \eqref{eq:expfkrho} we have
\begin{equation}\label{eq:auxf}
\widetilde{f}_{0}^{(\rho)}=\prod_{\nu=1}^{p}\widetilde{\varphi}_{\nu}^{(l(\rho))}=\frac{1}{\widetilde{\varphi}_{0}^{(l(\rho))}},\quad\mbox{for any}\,\,\rho\in\mathbb{Z}.
\end{equation}
We conclude from \eqref{eq:auxquotF} and \eqref{eq:auxf} that
\begin{equation}\label{eq:auxfinal}
\prod_{m=m_1}^{m_2-1}\frac{\widetilde{\varphi}_{0}^{(l(\rho+m(p+1)))}}{\widetilde{\varphi}_{0}^{(l(\rho+m(p+1)+1))}}\equiv 1.
\end{equation}
The reader can easily check that for any $m$,
\[
l(\rho+m(p+1)+1)=l(\rho+(m+1)(p+1)),
\]
hence \eqref{eq:auxfinal} reduces to
\[
\frac{\widetilde{\varphi}_{0}^{(l_1)}}{\widetilde{\varphi}_{0}^{(l_2)}}\equiv 1,\qquad l_1=l(\rho+m_1(p+1)),\quad l_2=l(\rho+(m_2-1)(p+1)+1).
\]
It is easily seen that the values $l_{1}$ and $l_{2}$ above are different, which contradicts the property described at the beginning of the proof.\end{proof}

\section{Formulae for $\widetilde{F}_{k}^{(\rho)}$ and   $a^{(\rho)}$}

\subsection{The limits $a^{(\rho)}$ and the normalizing constants $c_{k}^{(\rho)}$}
\begin{proposition}
Let $0\leq \rho\leq p(p+1)-1$. If $\rho \equiv k \mod p$, $0\leq k\leq p-1$, then
\begin{equation}\label{eq:relarhockr}
a^{(\rho)}=\prod_{j=1}^{p}\frac{c_{k+1}^{(\rho+j)}}{c_{0}^{(\rho+j)}\,c_{k}^{(\rho+j)}},\qquad c_{p}^{(\rho)}=1,
\end{equation}
where the constants $\{c_{k}^{(\rho)}\}_{k=0}^{p-1}$ are the positive constants that appear in the relation $F_{k}^{(\rho)}=c_{k}^{(\rho)} \widetilde{F}_{k}^{(\rho)}$ (see Proposition~\ref{prop:boundary}), and they are obtained solving the system of equations (6.30) in \cite{LopLopstar}.
\end{proposition}
\begin{proof}
Let $H_{n,k}=\frac{P_{n,k-1} \psi_{n,k}}{P_{n,k}}$, and let $\sigma_{n,k}$ be the measure defined in \eqref{varymeas:sigma}. Set
\[ K_{n,k}^{-2} := \int_{\Delta_{k}}P_{n,k}^2 \frac{|H_{n,k}|}{|P_{n,k-1} P_{n,k+1}|}\,d |\sigma_{n,k}|.
\]
In \cite[Theorem 3.5]{LopMin} it was proved that for $n\equiv k \mod p$, $0\leq k\leq p-1$, $n\geq p$, the recurrence coefficient $a_{n}$ satisfies
\[
a_{n}=\frac{K_{n-p,k}^{2}}{K_{n,k}^{2}}.
\]
As in \cite{LopLopstar}, we define the constants
\begin{equation}\label{def:kappajrho}
\kappa_{j}^{(\rho)}:=\frac{c_{j}^{(\rho)}}{(c_{j-1}^{(\rho)} c_{j+1}^{(\rho)})^{1/2}},\qquad 0\leq j\leq p-1,
\end{equation}
where by definition $c_{-1}^{(\rho)}=c_{p}^{(\rho)}=1$. Fix $\rho\in[0:p(p+1)-1]$ with $\rho\equiv k\mod p$. Using formulas (3.8) and (3.10) from \cite{LopLopstar}, we obtain
\begin{align*}
a^{(\rho)} & =\lim_{\lambda\rightarrow\infty} a_{\lambda p(p+1)+\rho}=\lim_{\lambda\rightarrow\infty}\prod_{j=0}^{p-1} \frac{K_{\lambda p(p+1)+\rho-p+j,k}^2}{K_{\lambda p(p+1)+\rho-p+j+1,k}^2} \\
& = \prod_{j=0}^{p-1}\frac{1}{(\kappa_{0}^{(\rho-p+j)}\cdots \kappa_{k}^{(\rho-p+j)})^2}=\prod_{j=0}^{p-1}\frac{c_{k+1}^{(\rho-p+j)}}{c_{0}^{(\rho-p+j)} c_{k}^{(\rho-p+j)}}.
\end{align*}
(For each $k$, the values of $\kappa_{k}^{(\rho)}$ and $c_{k}^{(\rho)}$ are defined periodically with period $p(p+1)$ in the parameter $\rho$.) From \eqref{eq:idprodFkrho} we deduce that
\[
\prod_{j=0}^{p-1}\frac{c_{k+1}^{(\rho-p+j)}}{c_{0}^{(\rho-p+j)} c_{k}^{(\rho-p+j)}}=\prod_{j=1}^{p}\frac{c_{k+1}^{(\rho+j)}}{c_{0}^{(\rho+j)}\,c_{k}^{(\rho+j)}}
\]
which concludes the proof of \eqref{eq:relarhockr}.
\end{proof}

\subsection{The quotients $\widetilde{F}_{k}^{(\rho)}/\widetilde{F}_{k-1}^{(\rho)}$}

\begin{theorem}\label{lem:quotFkr}
For each $0\leq \rho\leq p(p+1)-1$ and $1\leq k\leq p$, we have for $z\in\mathbb{C}\setminus(\Delta_{k-1}\cup\Delta_{k})$,
\begin{align}
\frac{F_{k}^{(\rho)}(z)}{F_{k-1}^{(\rho)}(z)} & =\frac{\xi_{k}^{(\rho)}(z)\,c_0^{(\rho)}}{\varepsilon_k^{(\rho)}}\frac{1}{1 + a^{(\rho)}\, \omega_{l}^{-1}\,\varphi_{k}^{(l)}(z)}\label{eq:ratioFalg}\\
\frac{\widetilde{F}_{k}^{(\rho)}(z)}{\widetilde{F}_{k-1}^{(\rho)}(z)} & =\frac{\xi_{k}^{(\rho)}(z)}{\varepsilon_k^{(\rho)}}\frac{c_{0}^{(\rho)}c_{k-1}^{(\rho)}}{c_{k}^{(\rho)}}\frac{1}{1 + a^{(\rho)}\, \omega_{l}^{-1}\,\varphi_{k}^{(l)}(z)}\label{eq:ratioFalg:2}
\end{align}
where $l=l(\rho)$ is the integer satisfying the conditions $l-1\equiv \rho\mod p$ and $1\leq l\leq p$, $\omega_l$ is defined in \eqref{def:omegal},
\[
\xi_{k}^{(\rho)}(z)=\begin{cases}
z & \mbox{if}\,\,\rho\equiv (k-1) \mod (p+1),\\
1 & \mbox{otherwise},
\end{cases}
\]
the constants $c_{k}^{(\rho)}$ are those that appear in the relation $F_{k}^{(\rho)}=c_{k}^{(\rho)} \widetilde{F}_{k}^{(\rho)}$ (see Proposition~\ref{prop:boundary}), and $\varepsilon_{k}^{(\rho)}$ is the constant (taking only the values $1$ or $-1$) given in \eqref{eq:identityekr} (see the Appendix).
\end{theorem}
\begin{proof}
Let $\rho\in[0:p(p+1)-1]$ be fixed. As indicated, $c_k^{(\rho)}$, $0\leq k\leq p-1$ are the positive constants for which $F_k^{(\rho)}= c_k^{(\rho)}\widetilde{F}_k^{(\rho)}$. We also define $c_{-1}^{(\rho)}=c_{p}^{(\rho)}=1$. Let $\kappa_{j}^{(\rho)}$ be the constant defined in \eqref{def:kappajrho}. Note that
\begin{equation}\label{eq:prodkappa}
(\kappa_0^{(\rho)}\cdots\kappa_{k-1}^{(\rho)})^2 = \frac{c_0^{(\rho)}c_{k-1}^{(\rho)}}{c_k^{(\rho)}}, \qquad k=1,\ldots,p.
\end{equation}

Combining \eqref{eq:prodkappa} and (3.11) in \cite{LopLopstar}, for $k=1,\ldots,p$ and $z\in\mathbb{C}\setminus(\Delta_{k-1}\cup\Delta_{k}\cup\{0\})$ we obtain
\begin{equation}\label{eq:ratiopsifunc}
\lim_{\lambda \to \infty} \frac{\psi_{\lambda p(p+1) + \rho +1,k}(z)}{\psi_{\lambda p(p+1) + \rho,k}(z)} = \frac{\varepsilon_k^{(\rho)}h_k^{(\rho)}(z)}{(\kappa_0^{(\rho)}\cdots\kappa_{k-1}^{(\rho)})^2}\frac{\widetilde{F}_k^{(\rho)}(z)}{\widetilde{F}_{k-1}^{(\rho)}(z)} = \frac{\varepsilon_k^{(\rho)}h_k^{(\rho)}(z)}{c_0^{(\rho)}}\frac{ {F}_k^{(\rho)}(z)}{{F}_{k-1}^{(\rho)}(z)}
\end{equation}
$(\widetilde{F}_{p}^{(\rho)}\equiv 1)$ where
\begin{equation}\label{eq:exphkrho}
h_k^{(\rho)}(z)=\begin{cases}
z & \mbox{if}\,\,\rho \equiv p \mod (p+1), \\
z^{-1} & \mbox{if}\,\,\rho \equiv (k-1) \mod (p+1), \\
1 & \mbox{otherwise}.
\end{cases}
\end{equation}

Assume that $\rho\equiv p \mod (p+1)$. Taking $n=\lambda p(p+1)+\rho$ in \eqref{threetermrecsecondkindpequena2} and using \eqref{eq:ratiopsifunc}, we get
\begin{align}\label{eq:chain:1}
\frac{c_0^{(\rho)}}{\varepsilon_k^{(\rho)}h_k^{(\rho)}(z)}\frac{z {F}_{k-1}^{(\rho)}(z)}{ {F}_k^{(\rho)}(z)} & = \lim_{\lambda \to \infty} \frac{z \psi_{\lambda p(p+1)+ \rho,k}(z)}{\psi_{\lambda p(p+1) + \rho +1,k}(z)}\notag\\
& =1 + \lim_{\lambda \to \infty}a_{\lambda p(p+1) + \rho} \frac{\psi_{\lambda p(p+1) + \rho-p,k}(z)}{\psi_{\lambda p(p+1) + \rho+1,k}(z)}\notag\\
& =1 + a^{(\rho)} \lim_{\lambda \to \infty}  \prod_{j=0}^p \frac{\psi_{\lambda p(p+1) + \rho-p+j,k}(z)}{\psi_{\lambda p(p+1) + \rho-p + j +1,k}(z)}\notag\\
& =1 + a^{(\rho)}\prod_{j=0}^p\frac{c_0^{(\rho-p+j)}}{\varepsilon_k^{(\rho-p+j)}h_k^{(\rho-p+j)}(z)}
 \frac{ {F}_{k-1}^{(\rho-p+j)}(z)}{ {F}_k^{(\rho-p+j)}(z)}\notag\\
& =1 + a^{(\rho)}\frac{ {f}_{k-1}^{(\rho)}(z)}{ {f}_k^{(\rho)}(z)}\prod_{j=0}^p\frac{c_0^{(\rho-p+j)}}{\varepsilon_k^{(\rho-p+j)}h_k^{(\rho-p+j)}(z)}.
\end{align}
In the last equality we use the identity $f_{k}^{(\rho-p)}=f_{k}^{(\rho)}$. Taking account of \eqref{eq:exphkrho} in the case $\rho\equiv p \mod (p+1)$, from \eqref{eq:chain:1} we obtain the relation
\begin{equation}\label{eq:caso1}
\frac{c_0^{(\rho)}}{\varepsilon_k^{(\rho)} }\frac{ {F}_{k-1}^{(\rho)}(z)}{ {F}_k^{(\rho)}(z)} =  1 + a^{(\rho)}\frac{ {f}_{k-1}^{(\rho)}(z)}{ {f}_k^{(\rho)}(z)}\prod_{j=0}^p\frac{c_0^{(\rho-p+j)}}{\varepsilon_k^{(\rho-p+j)}}.
\end{equation}

If $\rho\not\equiv p \mod (p+1)$, proceeding in the same fashion but using \eqref{threetermrecsecondkindpequena1} instead of \eqref{threetermrecsecondkindpequena2} we obtain the formulas
\begin{equation}\label{eq:casos}
1 + a^{(\rho)}\frac{ {f}_{k-1}^{(\rho)}(z)}{ {f}_k^{(\rho)}(z)}\prod_{j=0}^p
\frac{c_0^{(\rho-p+j)}}{\varepsilon_k^{(\rho-p+j)}} =\begin{cases}
\frac{c_0^{(\rho)}}{\varepsilon_k^{(\rho)}}\frac{z {F}_{k-1}^{(\rho)}(z)}{ {F}_k^{(\rho)}(z)}& \mbox{if}\,\,\rho\equiv (k-1) \mod (p+1), \\[1em]
\frac{c_0^{(\rho)}}{\varepsilon_k^{(\rho)} }\frac{ {F}_{k-1}^{(\rho)}(z)}{ {F}_k^{(\rho)}(z)} & \mbox{otherwise}.
\end{cases}
\end{equation}
In the second part of \eqref{eq:casos} one can include \eqref{eq:caso1}. In virtue of \eqref{eq:expfkrho}, we have $f_{k-1}^{(\rho)}/f_{k}^{(\rho)}=\mathrm{sg}(\varphi_{k}^{(l)}(\infty))\,\varphi_{k}^{(l)}$. So \eqref{eq:casos} is equivalent to
the identity
\begin{equation}\label{eq:summary}
\frac{F_{k}^{(\rho)}(z)}{F_{k-1}^{(\rho)}(z)}=\frac{\xi_{k}^{(\rho)}(z)\,c_0^{(\rho)}}{\varepsilon_k^{(\rho)}}\left(1 + a^{(\rho)}\frac{ {\varphi}_{k}^{(l)}(z)}{\mathrm{sg}({\varphi}_{k}^{(l)}(\infty))}\prod_{j=0}^p
\frac{c_0^{(\rho-p+j)} }{\varepsilon_k^{(\rho-p+j)}}\right)^{-1}
\end{equation}
which we shall simplify now.

First, we have the relation
\begin{equation}\label{eq:simpli1}
\omega_{l}^{-1}=\prod_{j=0}^{p}c_{0}^{(\rho-p+j)}
\end{equation}
where $\omega_{l}$ is defined in \eqref{def:omegal}. Indeed, in virtue of \eqref{eq:perfkrho} and \eqref{eq:expfkrho}, we have
\[
f_{0}^{(\rho-p)}=f_{0}^{(\rho)}=\mathrm{sg}\left(\prod_{\nu=1}^{p}\varphi_{\nu}^{(l)}(\infty)\right)\prod_{\nu=1}^{p}\varphi_{\nu}^{(l)}=\frac{1}{\mathrm{sg}(\varphi_{0}^{(l)}(\infty))}\frac{1}{\varphi_{0}^{(l)}}=\frac{1}{\varphi_{0}^{(l)}}
\]
and identifying the leading coefficients in the Laurent expansion at infinity of the extreme functions in this identity, we obtain \eqref{eq:simpli1}. We also have the identity
\begin{equation}\label{eq:simpli2}
\mathrm{sg}(\varphi_{k}^{(l)}(\infty)) \prod_{j=0}^{p}\varepsilon_{k}^{(\rho-p+j)}=1,
\end{equation}
which is proved in the Appendix (see Lemma~\ref{lem:propepsilon}).

Now, \eqref{eq:ratioFalg} and \eqref{eq:ratioFalg:2} follow from \eqref{eq:summary}, \eqref{eq:simpli1} and \eqref{eq:simpli2}, and the relation $F_{k}^{(\rho)}=c_{k}^{(\rho)} \widetilde{F}_{k}^{(\rho)}$.
\end{proof}

\subsection{Description of the functions $\widetilde{F}_{k}^{(\rho)}$}

\begin{theorem}\label{theo:descFkrconf}
Let $0\leq \rho\leq p(p+1)-1$ be fixed, and let $(k(\rho),l(\rho))$ be the unique pair of integers satisfying the conditions $0\leq k(\rho)\leq p$, $\rho\equiv (k(\rho)-1) \mod (p+1)$, and $1\leq l(\rho)\leq p$, $\rho\equiv (l(\rho)-1) \mod p$. For each $k=0,\ldots,p-1,$
\begin{equation}\label{eq:formFkrho}
\widetilde{F}_{k}^{(\rho)}(z)=\begin{cases}
C_{k}^{(\rho)} \prod_{j=0}^{k} (1+a^{(\rho)}\,\omega_{l(\rho)}^{-1}\,\varphi_{j}^{(l(\rho))}(z))^{-1} & \mbox{if}\,\,\,\,0\leq k<k(\rho),\\[1em]
z\,C_{k}^{(\rho)} \prod_{j=0}^{k}(1+a^{(\rho)}\,\omega_{l(\rho)}^{-1}\,\varphi_{j}^{(l(\rho))}(z))^{-1} & \mbox{if}\,\,\,\,k(\rho)\leq k\leq p-1,
\end{cases}
\end{equation}
where
\begin{equation}\label{eq:normconstant}
C_{k}^{(\rho)}=\begin{cases}
1, & k=0,\\[0.5em]
\prod_{j=1}^{k}\left(1+a^{(\rho)}\,\frac{\omega_{l(\rho),j}}{\omega_{l(\rho)}}\right), & 1\leq k\leq l(\rho)-1,\\[0.5em]
a^{(\rho)}\,\frac{\omega_{l(\rho),l(\rho)}}{\omega_{l(\rho)}}\,\prod_{j=1, j\neq l(\rho)}^{k}\left(1+a^{(\rho)}\,\frac{\omega_{l(\rho),j}}{\omega_{l(\rho)}}\right), & l(\rho)\leq k\leq p-1,
\end{cases}
\end{equation}
see \eqref{def:omegalj}. The constant $C_{k}^{(\rho)}$ has the following alternative expression:
\begin{equation}\label{eq:alt:Ckr}
C_{k}^{(\rho)}=\begin{cases}
1 & k=0,\\[0.2em]
(c_{k}^{(\rho)})^{-1} (c_{0}^{(\rho)})^{k+1} & 1\leq k\leq p-1,\,\,k\,\,\mbox{odd},\\[0.2em]
\varepsilon_{k}^{(\rho)}\,(c_{k}^{(\rho)})^{-1} (c_{0}^{(\rho)})^{k+1} & 1\leq k\leq p-1,\,\,k\,\,\mbox{even},
\end{cases}
\end{equation}
where $\varepsilon_{k}^{(\rho)}$ is defined in \eqref{eq:identityekr}.

In particular, we have
\begin{equation}\label{eq:descFzerorho}
\widetilde{F}_{0}^{(\rho)}(z)=\begin{cases}
(1+a^{(\rho)}\,\widetilde{\varphi}_{0}^{(l(\rho))}(z))^{-1} &  \mbox{if}\,\,\,\,\rho\not\equiv p \mod (p+1),\\[1em]
z\,(1+a^{(\rho)}\,\widetilde{\varphi}_{0}^{(l(\rho))}(z))^{-1} & \mbox{if}\,\,\,\,\rho\equiv p \mod (p+1).
\end{cases}
\end{equation}
\end{theorem}
\begin{proof}
Since $\widetilde{F}_{-1}^{(\rho)}\equiv 1$, formula \eqref{eq:descFzerorho} can be viewed as a particular case of \eqref{eq:ratioFalg:2} for $k=0$. Formula \eqref{eq:descFzerorho} is quite straightforward so let us prove it first. Assume $\rho\equiv p \mod (p+1)$. Applying \eqref{eq:relarhoF:1} we obtain
\[
\widetilde{F}_{0}^{(\rho)}(z)=z-\frac{a^{(\rho)}}{\prod_{i=\rho-p}^{\rho-1}\widetilde{F}_{0}^{(i)}(z)}=z-\frac{a^{(\rho)} \widetilde{F}_{0}^{(\rho)}(z)}{\prod_{i=\rho-p}^{\rho}\widetilde{F}_{0}^{(i)}(z)}.
\]
From \eqref{def:fkrho} and \eqref{eq:perfkrho} we deduce that $\widetilde{f}_{0}^{(\rho)}(z)=\widetilde{f}_{0}^{(\rho-p)}(z)=\prod_{i=\rho-p}^{\rho}\widetilde{F}_{0}^{(i)}(z)$, and therefore
\[
\widetilde{F}_{0}^{(\rho)}(z)=\frac{z}{1+\frac{a^{(\rho)}}{\widetilde{f}_{0}^{(\rho)}(z)}}=\frac{z}{1+a^{(\rho)}\,\widetilde{\varphi}_{0}^{(l)}(z)}=\frac{z}{1+a^{(\rho)}\,\omega_{l}^{-1}\,\varphi_{0}^{(l)}(z)},\qquad l=l(\rho),
\]
where we have applied \eqref{eq:expfkrho} and \eqref{normconfmap} in the second equality (note that $\prod_{\nu=0}^{p}\widetilde{\varphi}_{\nu}^{(l)}\equiv 1$.) The other identity in \eqref{eq:descFzerorho} is obtained in the same manner, starting from the relation \eqref{eq:relarhoF:2}.

Formula \eqref{eq:formFkrho} for $k=0$ follows from \eqref{eq:descFzerorho} and the first equality in \eqref{eq:normconstant}. For the rest of the values of $k$, it suffices to observe  that taking telescopic products from \eqref{eq:ratioFalg:2} it follows that $\widetilde{F}_k^{(\rho)}$ is a constant multiple of the function
\[
\chi_{k}^{(\rho)}(z):=\prod_{j=0}^{k}(1+a^{(\rho)}\,\omega_{l(\rho)}^{-1}\,\varphi_{j}^{(l(\rho))}(z))^{-1}
\]
when $0 \leq k<k(\rho)$, or a constant multiple of  $z\chi_k^{(\rho)}(z)$ for $k(\rho) \leq k \leq p-1$, i.e.,
\begin{equation}\label{eq:relFkrchikr}
\widetilde{F}_{k}^{(\rho)}(z)=\begin{cases}
C_{k}^{(\rho)}\chi_{k}^{(\rho)}(z), & 0\leq k<k(\rho),\\[0.5em]
z\,C_{k}^{(\rho)}\chi_{k}^{(\rho)}(z), & k(\rho)\leq k\leq p-1,
\end{cases}
\end{equation}
for some constant $C_{k}^{(\rho)}$. This constant must be such that the leading coefficient of the Laurent expansion of $C_{k}^{(\rho)}\chi_{k}^{(\rho)}$ at $\infty$ is $1$.

Let us determine the constant $C_{k}^{(\rho)}$. Consider the conformal function $\eta^{(\rho)}:\mathcal{R}\longrightarrow\overline{\mathbb{C}}$ defined in \eqref{def:etarhoconf}. Recall that the unique pole of this function is located at $0^{(k(\rho))}$. Consequently,
\[
C_{k}^{(\rho)}=\begin{cases}
1, & k=0\\
\prod_{j=1}^{k}\left(1+a^{(\rho)}\,\omega_{l(\rho)}^{-1}\,\varphi_{j}^{(l(\rho))}(\infty)\right), & 1\leq k\leq l(\rho)-1,\\
a^{(\rho)}\,\omega_{l(\rho)}^{-1}\,\left(\varphi_{l(\rho)}^{(l(\rho))}\right)'(\infty) \prod_{j=1, j\neq l(\rho)}^{k}(1+a^{(\rho)} \omega_{l(\rho)}^{-1} \varphi_{j}^{(l(\rho))}(\infty)), & l(\rho)\leq k\leq p-1,
\end{cases}
\]
which is \eqref{eq:normconstant}. From \eqref{eq:ratioFalg:2} and \eqref{eq:relFkrchikr} we obtain that for $1\leq k\leq p-1$,
\[
C_{k}^{(\rho)}=\prod_{j=1}^{k}\frac{\varepsilon_{j}^{(\rho)} c_{0}^{(\rho)} c_{j-1}^{(\rho)}}{c_{j}^{(\rho)}}=\frac{(c_{0}^{(\rho)})^{k+1}}{c_{k}^{(\rho)}}\,\prod_{j=1}^{k}\varepsilon_{j}^{(\rho)}
\]
and applying \eqref{eq:simpli3bis} we get \eqref{eq:alt:Ckr}.
\end{proof}

\begin{remark}
In terms of the function $\eta^{(\rho)}$ defined in \eqref{def:etarhoconf}, formula \eqref{eq:formFkrho} admits the form
\[
\widetilde{F}_{k}^{(\rho)}(z)=\begin{cases}
\prod_{j=0}^{k}\widetilde{\eta}_{j}^{(\rho)}(z), & \mbox{if}\,\,\,\,0\leq k<k(\rho),\\[1em]
z\prod_{j=0}^{k}\widetilde{\eta}_{j}^{(\rho)}(z), & \mbox{if}\,\,\,\,k(\rho)\leq k\leq p-1,
\end{cases}
\]
for $z \in \mathbb{C} \setminus \Delta_k$.\end{remark}

\subsection{Proof of $4)$ in Theorem \ref{theo:main:2}}
\begin{proof}
Let $\rho\in[0:p(p+1)-1]$, and let $(k,l)$ be the pair of parameters indicated in the statement of the result we are proving. With these values, formula \eqref{eq:ratioFalg} establishes that
\begin{equation} \label{relation}
\frac{F_{k}^{(\rho)}(z)}{F_{k-1}^{(\rho)}(z)} =\frac{z\,c_0^{(\rho)}}{\varepsilon_k^{(\rho)}}\frac{1}{1 + a^{(\rho)}\, \omega_{l}^{-1}\,\varphi_{k}^{(l)}(z)}, \qquad z \in \mathbb{C} \setminus (\Delta_{k-1} \cup \Delta_k).
\end{equation}
Recall that $F_{-1}^{(\rho)}\equiv 1$, so if $k=0$, then \eqref{relation} is understood to be
\[
F_{0}^{(\rho)}(z)=\frac{c_{0}^{(\rho)}z}{1+a^{(\rho)}\,\omega_{l}^{-1}\,\varphi_{0}^{(l)}(z)},\qquad z\in\mathbb{C}\setminus\Delta_{0},
\]
see the second relation in \eqref{eq:descFzerorho}.

Assume first that $0 \not\in \Delta_{k-1} \cup \Delta_k$. Then \eqref{eq:descrip:arho} follows immediately because the left side of \eqref{relation} must be different from zero when $z=0$ so the denominator in the right side of \eqref{relation} must vanish at the origin.

Now assume that $0\in\Delta_{k-1} \cup \Delta_k$. In this case, by definition of the intervals $\Delta_j$, $0$ must be an extreme point of either $\Delta_{k-1}$ or $\Delta_k$. In \eqref{relation}, take the square of the absolute value and make $z$ tend to $x \in \Delta_{k-1} \cup \Delta_k$. By continuity, we obtain
\begin{equation} \label{rel2}
\left|\frac{F_{k}^{(\rho)}(x_\pm)}{F_{k-1}^{(\rho)}(x_\pm)} \right|^2 =  \frac{|x\,c_0^{(\rho)}|^2}{|1 + a^{(\rho)}\, \omega_{l}^{-1}\,\varphi_{k}^{(l)}(x_{\pm})|^2}, \qquad x \in \Delta_{k-1} \cup \Delta_k,
\end{equation}
where $x_\pm$ is either the limiting point on $\Delta_{k-1} \cup \Delta_k$ from above or below. It does not matter which limit you take so we will simply write $x$.

Assume that $0 \in \Delta_k$. Then $0 \not\in \Delta_{k-1} \cup \Delta_{k+1}$ and, therefore, $F_{k+1}^{(\rho)}(0) \neq 0,  F_{k-1}^{(\rho)}(0) \neq \infty$. When $k=0$ then $\Delta_{-1} = \emptyset$ and $F_{-1}^{(\rho)}\equiv 1$. Taking account of \eqref{eq:bv:lnp:3} or \eqref{eq:bv:lp:1} (the latter in the case when $\rho\equiv -1 \mod (p+1)$, or what is the same $\rho\equiv p \mod (p+1)$) combined with
\eqref{rel2} it follows that
\begin{equation} \label{rel3} \frac{|F_{k}^{(\rho)}(x)|^2}{|x F_{k-1}^{(\rho)}(x)|}  = |F_{k+1}^{(\rho)}(x)| =  \frac{|xF_{k-1}^{(\rho)}(x)|\,(c_0^{(\rho)})^2}{|1 + a^{(\rho)}\, \omega_{l}^{-1}\,\varphi_{k}^{(l)}(x)|^2}, \qquad x \in \Delta_k\setminus\{0\}.
\end{equation}
Now, making $x\to 0$  in \eqref{rel3} we conclude that $1 + a^{(\rho)}\, \omega_{l}^{-1}\,\varphi_{k}^{(l)}(0) = 0$, which implies \eqref{eq:descrip:arho}. If $k=0$ we have concluded.

Finally, suppose that $0 \in \Delta_{k-1}, k =1,\ldots, p$. Then $0 \not\in \Delta_{k-2} \cup \Delta_{k}$.  Using \eqref{eq:bv:lnp:2} with $k$ replaced with $k-1$ it follows that
\[\frac{|x||F_{k-1}^{(\rho)}(x)|^2}{|F_{k-2}^{(\rho)}(x)F_{k}^{(\rho)}(x)|} = 1, \qquad x \in \Delta_{k-1} \setminus \{0\}.
\]
where $F_{-1}^{(\rho)} \equiv 1$ when $k=1$. This relation combined with \eqref{rel2} gives
\[ \frac{|x||F_{k}^{(\rho)}(x)|}{|F_{k-2}^{(\rho)}(x)|} = \frac{|F_{k}^{(\rho)}(x)|^2}{|F_{k-1}^{(\rho)}(x)|^2} =  \frac{|x\,c_0^{(\rho)}|^2}{|1 + a^{(\rho)}\, \omega_{l}^{-1}\,\varphi_{k}^{(l)}(x)|^2}, \qquad x \in \Delta_{k-1} \setminus \{0\}.
\]
Cancelling out the common factor $|x|$ and letting $x\to 0$ it follows that  $1 + a^{(\rho)}\, \omega_{l}^{-1}\,\varphi_{k}^{(l)}(0) =0$ because $F_{k}^{(\rho)}(0)/F_{k-2}^{(\rho)}(0)\neq 0$. With this we conclude the proof.
\end{proof}

\subsection{Proof of Theorem \ref{theo:main:1}}
\begin{proof}
Formula \eqref{eq:ratioasympQ} is a consequence of \eqref{eq:ratioQn:1}--\eqref{eq:ratioQn:2} and \eqref{eq:descFzerorho}.

According to (3.13) in \cite{LopLopstar}, we have
\[
\lim_{\lambda\rightarrow\infty}\frac{\Psi_{\lambda p(p+1)+\rho+1,k}(z)}{\Psi_{\lambda p(p+1)+\rho,k}(z)}=\frac{\varepsilon_{k}^{(\rho)}\,g_{k}^{(\rho)}(z)}{(\kappa_{0}^{(\rho)}\cdots \kappa_{k-1}^{(\rho)})^2}\,\frac{\widetilde{F}_{k}^{(\rho)}(z^{p+1})}{\widetilde{F}_{k-1}^{(\rho)}(z^{p+1})},\qquad z\in\mathbb{C}\setminus(\Gamma_{k-1}\cup\Gamma_{k}\cup\{0\}),
\]
where
\[
g_{k}^{(\rho)}(z)=\begin{cases}
z^{-p} & \mbox{if}\,\,\rho\equiv (k-1) \mod (p+1),\\
z & \mbox{otherwise}.
\end{cases}
\]
Applying now \eqref{eq:ratioFalg:2} and \eqref{eq:prodkappa}, we have
\[
\frac{\varepsilon_{k}^{(\rho)}\,g_{k}^{(\rho)}(z)}{(\kappa_{0}^{(\rho)}\cdots \kappa_{k-1}^{(\rho)})^2}\,\frac{\widetilde{F}_{k}^{(\rho)}(z^{p+1})}{\widetilde{F}_{k-1}^{(\rho)}(z^{p+1})}=\frac{\xi_{k}^{(\rho)}(z^{p+1})\,g_{k}^{(\rho)}(z)}{1+a^{(\rho)}\,\omega_{l}^{-1}\,\varphi_{k}^{(l)}(z^{p+1})}=\frac{z}{1+a^{(\rho)}\,\omega_{l}^{-1}\,\varphi_{k}^{(l)}(z^{p+1})},
\]
so \eqref{eq:ratioasympPsink} is justified.
\end{proof}

We also obtain a result similar to Theorem~\ref{theo:main:1} for the functions $\psi_{n,k}$.

\begin{theorem}
Under the same assumptions as in Theorem~\ref{theo:main:1}, for each $0\leq \rho\leq p(p+1)-1$ and $1\leq k\leq p$ we have
\begin{equation}\label{eq:ratiolittlepsink}
\lim_{\lambda\rightarrow\infty}\frac{\psi_{\lambda p(p+1)+\rho+1,k}(z)}{\psi_{\lambda p(p+1)+\rho,k}(z)}=\begin{cases}
\frac{z}{1+a^{(\rho)}\,\omega_{l}^{-1}\,\varphi_{k}^{(l)}(z)} & \mbox{if}\,\,\rho\equiv p \mod (p+1),\\[1em]
\frac{1}{1+a^{(\rho)}\,\omega_{l}^{-1}\,\varphi_{k}^{(l)}(z)} & \mbox{otherwise},
\end{cases}
\end{equation}
uniformly on compact subsets of $\mathbb{C}\setminus(\Delta_{k-1}\cup\Delta_{k}\cup\{0\})$, where $l=l(\rho)$ is the integer satisfying the conditions $1\leq l\leq p$ and $l-1\equiv \rho\mod p$, and $\omega_{l}$ is the constant defined in \eqref{def:omegal}.
\end{theorem}
\begin{proof}
Formula (3.11) in \cite{LopLopstar} asserts that
\[
\lim_{\lambda\rightarrow\infty}\frac{\psi_{\lambda p(p+1)+\rho+1,k}(z)}{\psi_{\lambda p(p+1)+\rho,k}(z)}=\frac{\varepsilon_{k}^{(\rho)}\,h_{k}^{(\rho)}(z)}{(\kappa_{0}^{(\rho)}\cdots \kappa_{k-1}^{(\rho)})^2}\,\frac{\widetilde{F}_{k}^{(\rho)}(z)}{\widetilde{F}_{k-1}^{(\rho)}(z)},
\]
uniformly on compact subsets of $\mathbb{C}\setminus(\Delta_{k}\cup\Delta_{k-1}\cup\{0\})$, where $h_{k}^{(\rho)}(z)$ is indicated in \eqref{eq:exphkrho}. Applying now \eqref{eq:ratioFalg:2} and \eqref{eq:prodkappa}, we obtain that the limiting function is $\xi_{k}^{(\rho)}(z)\,h_{k}^{(\rho)}(z)\,(1+a^{(\rho)}\,\omega_{l}^{-1}\,\varphi_{k}^{(l)}(z))^{-1}$. The expression $\xi_{k}^{(\rho)}(z)\,h_{k}^{(\rho)}(z)$ equals $z$ if $\rho\equiv p \mod (p+1)$ and it equals $1$ otherwise.
\end{proof}

\subsection{Proof of $5)$ in Theorem \ref{theo:main:2}}
We include statement $5)$ of Theorem \ref{theo:main:2} as part of the following more general result:

\begin{theorem}
Assume that $0\in\Delta_{\overline{k}}$ for some $0\leq \overline{k}\leq p-1$. Then, for any $0\leq \overline{\rho}\leq p(p+1)-1$ such that $\overline{\rho}\equiv (\overline{k}-1) \mod (p+1)$, we have
\begin{equation}\label{eq:relararmp}
a^{(\overline{\rho}-p)}=a^{(\overline{\rho})}.
\end{equation}
Moreover, for any $0\leq k\leq p-1$,
\begin{equation}\label{eq:relFkrFkrmps}
\frac{\widetilde{F}_{k}^{(\overline{\rho})}(z)}{\widetilde{F}_{k}^{(\overline{\rho}-p)}(z)}\equiv\begin{cases}
1 & \mbox{if}\,\,\,\,k\neq \overline{k},\\
z & \mbox{if}\,\,\,\,k=\overline{k}.
\end{cases}
\end{equation}
If $0\notin\Delta_{k}$ for all $0\leq k\leq p-1,$ then for any $0\leq \rho\leq p(p+1)-1$, the set of $p+1$ values $\{a^{(\rho+mp)}\}_{m=0}^{p}$ is formed by distinct quantities.
\end{theorem}
\begin{proof}
For any $0\leq \rho\leq p(p+1)-1$, let $(k(\rho),l(\rho))$ be the unique pair of integers satisfying the conditions stated in Theorem~\ref{theo:main:2}.4.

Assume that $0\in\Delta_{\overline{k}}$ for some $0\leq \overline{k}\leq p-1$, and let $\overline{\rho}\in[0:p(p+1)-1]$ be such that $\overline{\rho}\equiv (\overline{k}-1) \mod (p+1)$. Since the sheets $\mathcal{R}_{\overline{k}}$ and $\mathcal{R}_{\overline{k}+1}$ of the Riemann surface are glued along the interval $\Delta_{\overline{k}}$, we have
\begin{equation}\label{eq:eqphis}
\varphi_{\overline{k}}^{(l)}(0)=\varphi_{\overline{k}+1}^{(l)}(0),\qquad \mbox{for all}\,\,1\leq l\leq p.
\end{equation}
We also have $\overline{k}=k(\overline{\rho})$, $k(\overline{\rho}-p)=k(\overline{\rho})+1=\overline{k}+1$, and $l(\overline{\rho}-p)=l(\overline{\rho})=:l$. Therefore, applying \eqref{eq:descrip:arho} and \eqref{eq:eqphis}, we obtain
\[
a^{(\overline{\rho}-p)}=-\frac{\omega_{l(\overline{\rho}-p)}}{\varphi_{k(\overline{\rho}-p)}^{(l(\overline{\rho}-p))}(0)}=-\frac{\omega_{l}}{\varphi_{\overline{k}}^{(l)}(0)}=a^{(\overline{\rho})}.
\]
This settles \eqref{eq:relararmp}.

Now we prove \eqref{eq:relFkrFkrmps}. First, observe that \eqref{eq:relararmp} and \eqref{eq:normconstant} easily imply that $C_{k}^{(\overline{\rho})}=C_{k}^{(\overline{\rho}-p)}$ for any $0\leq k\leq p-1$, because $l(\overline{\rho})=l(\overline{\rho}-p)$. Therefore, comparing the expressions of $\widetilde{F}_{k}^{(\overline{\rho})}$ and $\widetilde{F}_{k}^{(\overline{\rho}-p)}$ that \eqref{eq:formFkrho} gives, and taking into account that $k(\overline{\rho}-p)=\overline{k}+1=k(\overline{\rho})+1$, we see that $\widetilde{F}_{k}^{(\overline{\rho})}\equiv \widetilde{F}_{k}^{(\overline{\rho}-p)}$ for all $k\neq \overline{k}$, $0\leq k\leq p-1$, and $\widetilde{F}_{k}^{(\overline{\rho})}\equiv z\widetilde{F}_{k}^{(\overline{\rho}-p)}$ if $k=\overline{k}$. This settles \eqref{eq:relFkrFkrmps}.

Assume that $0\notin\Delta_{k}$ for all $0\leq k\leq p-1$, and let $\rho\in[0:p(p+1)-1]$ be fixed. Let $l$ be the corresponding integer satisfying $1\leq l\leq p$ and $\rho\equiv (l-1)\mod p$. Applying \eqref{eq:descrip:arho}, we find that
\[
\{a^{(\rho+mp)}: 0\leq m\leq p\}=\{-\omega_{l}/\varphi_{k}^{(l)}(0): 0\leq k\leq p\}
\]
Now, $\varphi^{(l)}:\mathcal{R}\longrightarrow\overline{\mathbb{C}}$ is a bijection, and the assumption on the intervals $\Delta_{k}$ ensures that the points at the origin in the different sheets $\mathcal{R}_{k}$, $0\leq k\leq p,$ represent different points on $\mathcal{R}$. Therefore the values $-\omega_{l}/\varphi_{k}^{(l)}(0)$, $0\leq k\leq p,$ are distinct.
\end{proof}

\section{Appendix}

The constants $\varepsilon_{k}^{(\rho)}$, $1\leq k\leq p$, taking the values $1$ or $-1$, arised first in our previous work \cite{LopLopstar}. They are defined in Remark 6.5 of \cite{LopLopstar}, but here we shall use the identity
\begin{equation}\label{eq:identityekr}
\varepsilon_{k}^{(\rho)}=(-1)^{Z(\rho+1,2\lceil (k-1)/2\rceil)-Z(\rho,2\lceil (k-1)/2\rceil)+\theta(\rho,k-1)}, \qquad 1\leq k\leq p,\,\,\rho\in\mathbb{Z}_{\geq 0},
\end{equation}
where $\lceil x\rceil=\min\{m\in\mathbb{Z}: m\geq x\}$, and for integers $n\geq 0$ and $0\leq k\leq p-1$,
\begin{equation}\label{def:thetank}
\theta(n,k):=\begin{cases}
1 & \mbox{if}\,\,\ell(n)\in[0:k-1],\\
0 & \mbox{if}\,\,\ell(n)\in[k+1:p-1],\\
1 & \mbox{if}\,\,\ell(n)=k,\,\,k\,\,\mbox{odd},\\
0 & \mbox{if}\,\,\ell(n)=k,\,\,k\,\,\mbox{even},\\
1 & \mbox{if}\,\,\ell(n)=p,
\end{cases}
\end{equation}
and $\ell(n)$ is the integer defined by the conditions $n\equiv \ell(n) \mod (p+1)$, $0\leq \ell(n)\leq p$. The identity \eqref{eq:identityekr} is immediately obtained from formula (6.35) and Lemma 4.3 in \cite{LopLopstar}.

\begin{lemma}\label{lem:propepsilon}
With $\rho, k, l$ as in Theorem~\ref{lem:quotFkr}, we have
\begin{equation}\label{eq:simpli2bis}
\mathrm{sg}(\varphi_{k}^{(l)}(\infty)) \prod_{j=0}^{p}\varepsilon_{k}^{(\rho-p+j)}=1.
\end{equation}
We also have
\begin{equation}\label{eq:simpli3bis}
\prod_{j=1}^{k} \varepsilon_{j}^{(\rho)}=\begin{cases}
1 & \mbox{if}\,\,k\,\,\mbox{is odd},\\
\varepsilon_{k}^{(\rho)} & \mbox{if}\,\,k\,\,\mbox{is even}.
\end{cases}
\end{equation}
\end{lemma}
\begin{proof}
Applying \eqref{eq:identityekr}, we have
\begin{equation}\label{eq:firstfirst}
\prod_{j=0}^{p}\varepsilon_{k}^{(\rho-p+j)}=(-1)^{\sum_{j=0}^{p}\left(Z(\rho-p+j+1,2\lceil (k-1)/2\rceil)-Z(\rho-p+j,2\lceil (k-1)/2\rceil)+\theta(\rho-p+j,k-1)\right)}.
\end{equation}
It follows from \eqref{eq:altformL} and the fact that $\Lambda(n,k)$ is periodic with period $p$ with respect to $n$, that
\begin{align}
\Lambda(\rho,2\lceil (k-1)/2\rceil) & =\Lambda(\rho-p,2\lceil (k-1)/2\rceil)\notag\\
& =\sum_{j=0}^{p}\left(Z(\rho-p+j+1,2\lceil (k-1)/2\rceil)-Z(\rho-p+j,2\lceil (k-1)/2\rceil)\right).\label{eq:relLambda}
\end{align}

Now we analyze the expression $(-1)^{\sum_{j=0}^{p}\theta(\rho-p+j,k-1)}$. For each fixed $k$, the function $\theta(n,k)$ is periodic with period $p+1$ with respect to $n$, i.e., $\theta(n,k)=\theta(n+p+1,k)$. Since $p+1$ is also the number of terms in the summation $\sum_{j=0}^{p}\theta(\rho-p+j,k-1)$ and the values $\rho-p+j$, $j=0,\ldots,p,$ are consecutive, we deduce that
\[
\sum_{j=0}^{p}\theta(\rho-p+j,k-1)=\sum_{n=0}^{p}\theta(n,k-1),\quad\mbox{for any}\,\,\rho.
\]
In view of \eqref{def:thetank}, we easily find that $\sum_{n=0}^{p}\theta(n,k-1)$ is always an odd integer  (it equals $k$ if $k$ is odd and it equals $k+1$ if $k$ is even). Hence,
\begin{equation}\label{eq:summinus}
(-1)^{\sum_{j=0}^{p}\theta(\rho-p+j,k-1)}=-1.
\end{equation}
We conclude from \eqref{eq:relLambda}, \eqref{eq:summinus} and \eqref{eq:firstfirst} that
\[
\prod_{j=0}^{p}\varepsilon_{k}^{(\rho-p+j)}=(-1)^{\Lambda(\rho,2\lceil (k-1)/2\rceil)+1}.
\]

To finish the proof of \eqref{eq:simpli2bis}, we show now that
\begin{equation}\label{eq:prodepsilon:1}
\mathrm{sg}(\varphi_{k}^{(l)}(\infty)) (-1)^{\Lambda(\rho,2\lceil (k-1)/2\rceil)+1}=1,
\end{equation}
where $l$ and $\rho$ are related as in Theorem~\ref{lem:quotFkr}. Assume first that $\rho\equiv s \mod p$, $0\leq s\leq 2\lceil(k-1)/2\rceil-1$. According to \eqref{eq:descLambda}, in this case $\Lambda(\rho,2\lceil (k-1)/2\rceil)=0$. By definition of $l$, we have $l-1\equiv s \mod p$, hence $l-1=s$ and $l\leq 2\lceil(k-1)/2\rceil$. This inequality and  \eqref{eq:signphikl:1}--\eqref{eq:signphikl:2} imply that $\mathrm{sg}(\varphi_{k}^{(l)}(\infty))=-1$, which proves \eqref{eq:prodepsilon:1} in this case. Assume now that $\rho\equiv s \mod p$ and $2\lceil(k-1)/2\rceil\leq s\leq p-1$. Then \eqref{eq:descLambda} gives $\Lambda(\rho,2\lceil (k-1)/2\rceil)=1$, and in this case $l\geq 2\lceil(k-1)/2\rceil+1$, so by \eqref{eq:signphikl:1}--\eqref{eq:signphikl:2} we get $\mathrm{sg}(\varphi_{k}^{(l)}(\infty))=1$. This proves \eqref{eq:prodepsilon:1}.

To prove \eqref{eq:simpli3bis}, it suffices to show that $\varepsilon_{1}^{(\rho)}=1$ and that for $3\leq k\leq p$ odd, we have $\varepsilon_{k-1}^{(\rho)}\varepsilon_{k}^{(\rho)}=1$. Indeed, applying \eqref{eq:identityekr}, \eqref{def:thetank} and \eqref{eq:formZn0}, we obtain
\[
\varepsilon_{1}^{(\rho)}=(-1)^{Z(\rho+1,0)-Z(\rho,0)+\theta(\rho,0)}=1.
\]
Let $3\leq k\leq p$ be odd. From \eqref{eq:identityekr} we obtain that $\varepsilon_{k-1}^{(\rho)}\,\varepsilon_{k}^{(\rho)}$ equals $-1$ raised to the expression
\begin{equation}\label{eq:power}
2 Z(\rho+1,2\lceil(k-2)/2\rceil)-2 Z(\rho,2\lceil(k-2)/2\rceil)+\theta(\rho,k-2)+\theta(\rho,k-1)
\end{equation}
where we used that $\lceil(k-2)/2\rceil=\lceil(k-1)/2\rceil=(k-1)/2$. The reader can easily check that because $k$ is odd, we also have $\theta(\rho,k-2)=\theta(\rho,k-1)$. Hence, \eqref{eq:power} is even.
\end{proof}

\smallskip

\noindent\textbf{Acknowledgements:} We thank the anonymous referees for their valuable comments.

\end{document}